\documentclass[twoside,a4paper,reqno,11pt]{amsart}
\usepackage{amsfonts, amsbsy, amsmath, amssymb, latexsym}
\usepackage{mathrsfs,array}
\usepackage[top=30mm,right=30mm,bottom=30mm,left=30mm]{geometry}
\usepackage{stmaryrd}
\usepackage{bm}
\usepackage{xcolor}
\usepackage{hyperref}

\newcommand{\F}{\mathbb{F}}
\newcommand{\Z}{\mathbb{Z}}
\newcommand{\N}{\mathbb{N}}
\newcommand{\R}{\mathbb{R}}

 \renewcommand{\to}{\rightarrow}

\DeclareMathOperator{\rk}{rk}

\DeclareMathOperator{\Gal}{Gal}
\DeclareMathOperator{\GL}{GL}
\DeclareMathOperator{\PGL}{PGL}

\DeclareMathOperator{\Max}{Max}

\makeatletter
\newcommand{\imod}[1]{\allowbreak\mkern4mu({\operator@font mod}\,\,#1)}
\makeatother

\newtheorem{theorem}{Theorem}
\newtheorem*{conj*}{Conjecture}
\newtheorem{lemma}[theorem]{Lemma}
\newtheorem{propn}[theorem]{Proposition}

\newtheorem{corollary}[theorem]{Corollary}

\newtheorem{thm}{Theorem}[section]

\newtheorem{cor}[thm]{Corollary}

\theoremstyle{definition}

\begin{document}

 \author{Damian Sercombe}
\address{D. Sercombe, Institute of Mathematics, Hebrew University, Jerusalem 91904, Israel}
\email{damian.sercombe@mail.huji.ac.il}

\author{Aner Shalev}
\address{A. Shalev, Institute of Mathematics, Hebrew University, Jerusalem 91904, Israel}
\email{shalev@math.huji.ac.il}

\title{Random generation of associative algebras}

\dedicatory{Dedicated to the memory of Peter Neumann}

\begin{abstract}

There has been considerable interest in recent decades in questions of random generation of finite and profinite groups, and finite simple groups in particular. In this paper we study similar notions for finite and profinite associative algebras. Let $k=\F_q$ be a finite field. Let $A$ be a finite dimensional, associative, unital algebra over $k$. Let $P(A)$ be the probability that two elements of $A$ chosen (uniformly and independently) at random will generate $A$ as a unital $k$-algebra. It is known that, if $A$ is simple, then $P(A) \to 1$ as $|A| \to \infty$. We extend this result to a large class of finite associative algebras. For $A$ simple, we find the optimal lower bound for $P(A)$ and we estimate the growth rate of $P(A)$ in terms of the minimal index $m(A)$ of any proper subalgebra of $A$. We also study the random generation of simple algebras $A$ by two elements that have a given characteristic polynomial (resp. a given rank). In addition, we bound above and below the minimal number of generators of general finite algebras. Finally, we let $A$ be a profinite algebra over $k$. We show that $A$ is positively finitely generated if and only if $A$ has polynomial maximal subalgebra growth. Related quantitative results are also established.
\end{abstract}

\subjclass[2020]{Primary 16P10; Secondary 15B52 \vspace{2mm}}

 \thanks{
\hspace{-6.25mm} Both authors are affiliated with the Institute of Mathematics, Hebrew University, Jerusalem 91904, Israel.\\[1ex]
DS was supported by a Post-Doctoral Fellowhip from ISF grant 686/17 of AS.
 AS was partially supported by ISF grant 686/17 and the Vinik Chair of mathematics which he holds.}

\maketitle

\section{Introduction}

In the past few decades there has been extensive research on random generation of finite and profinite groups with emphasis on finite simple groups. See for instance the survey articles \cite{L, Sh} and the references therein.

The study of random generation of associative algebras is less well developed. Consider the algebra $M_n(q)$ of $n \times n$ matrices over a finite field $\F_q$. In 1995 it was shown by Peter Neumann and Cheryl Praeger \cite{NP} that the probability that two matrices in $M_n(q)$, chosen independently under the uniform distribution, generate $M_n(q)$ as a $\F_q$-algebra tends to $1$ as $|M_n(q)| \to \infty$. See also the subsequent paper \cite{KMP} by Kravchenko, Mazur and Petrenko for additional results on random generation of finite and infinite algebras.

One can refine this problem and consider random generation of an algebra by two elements that satisfy a certain property. A matrix in $M_n(q)$ is \textit{cyclic} if its characteristic polynomial is equal to its minimal polynomial. Neumann and Praeger showed in \cite{NP} that almost all pairs of cyclic matrices in $M_n(q)$ will generate it as a $\F_q$-algebra.
Amongst other results of this flavour, we show that -- given a monic polynomial $f$ of degree $n$ over $\F_q$ -- almost all pairs of matrices in $M_n(q)$ with characteristic polynomial $f$ will generate it as a $\F_q$-algebra.

In this paper we study random generation of finite and profinite associative algebras, and we obtain some new results also in the case of simple algebras.

Let $k$ be a finite field, that is, $k=\F_q$ for some prime power $q$. Unless otherwise stated, all algebras in this paper are assumed to be over $k$, and are associative and unital. Subalgebras of a unital algebra are required to contain the multiplicative identity of the original algebra. We first focus on the study of finite algebras. Later on, we look at profinite algebras.

Let $A$ be an associative, unital, finite-dimensional algebra over $k$ (a.k.a. a \textit{finite algebra}). Let $A^{\times}$ denote the group of units of $A$. Let $A^N$ denote the set of nilpotent elements of $A$. The Jacobson radical $J(A)$ of $A$ is a nilpotent ideal of $A$. If $J(A)$ is trivial then $A$ is \textit{semisimple}.

In this paragraph we summarise the Wedderburn-Malcev Principal Theorem (Theorems $5.3.20$ and $5.3.21$ of \cite{R}). There exists a semisimple subalgebra $S$ of $A$ such that $A=S \oplus J(A)$ as vector spaces. If $S'$ is another subalgebra of $A$ satisfying $A=S' \oplus J(A)$ then $S'$ is conjugate to $S$ by an element of $1+J(A)$.
Wedderburn’s little theorem (Theorem $7.1.11$ of \cite{R}) states that all finite division algebras are fields. Combining this with another
theorem of Wedderburn (Theorem $2.1.8$ of \cite{R}), it follows that there is an algebra isomorphism $S \cong \prod_{i=1}^r M_{n_i}(q^{m_i})$
for some integers $r$, $n_1$, ..., $n_r$, $m_1$, ..., $m_r$ that is unique up to permutation of the factors.

Denote $n:=\min_{i=1,...,r}\{n_i\}$ and $m:=\min_{i=1,...,r}\{m_i\}$. Fix constants $c >1$ and $\lambda>0$. We say that $A$ is \textit{bounded by $(c,\lambda)$} if $r \leq \lambda c^{\min\{m,n\}/2}$ and $\dim J(A)/J(A)^2 \leq \log_q\lambda +{\min\{m,n\}^2}\log_q c$.


A subset $X$ of $A$ is a generating set if the set of all monomials in the elements of $X$ (including the trivial monomial) spans $A$ as a $k$-vector space. We define $P(A)$ to be the probability that two elements of $A$ chosen uniformly at random will generate $A$ as a (unital) $k$-algebra. That is,
$$P(A)=\frac{|\{(x,y) \in A \times A: \langle x,y\rangle=A\}|}{|A|^2}.$$

\begin{thm}\label{randomgeneration} Fix constants $1<c <q$ and $\lambda>0$. Let $A$ be a finite algebra, say $A = \big(\prod_{i=1}^r M_{n_i}(q^{m_i})\big) \oplus J(A)$, that is bounded by $(c,\lambda)$. Denote $n:=\min_{i=1,...,r}\{n_i\}$ and $m:=\min_{i=1,...,r}\{m_i\}$. Then $P(A) \to 1$ as $n \to \infty$, as $m \to \infty$ or as $q \to \infty$.
\end{thm}

It is not true in general that $P(A) \to 1$ as $|A| \to \infty$. For example, let $A$ be as in the theorem above and suppose there exists a positive integer $i \le r$ such that $n_i=1$ and $m_i = 2$. Then $A$ has a maximal subalgebra $B$ satisfying $A/B \cong k$. Hence $|B|/|A| = q^{-1}$, so $1-P(A) \geq |B|^2/|A|^2=q^{-2}$. Fixing $q$ and letting $|A|$ tend to infinity
we see that $P(A) \le 1-q^{-2}$ is bounded away from $1$.

Moreover, let $A=k^r$ for some $r \in \N$. Then any maximal subalgebra $B$ of $A$ has codimension $1$, and it is easy to see
that $P(A)\to 0$ as $r \to \infty$. However, Theorem \ref{randomgeneration} implies the following known result.

\begin{cor}\label{randomgenerationcor} Let $A$ be a finite simple algebra. Then $P(A) \to 1$ as $|A| \to \infty$.
\end{cor}

This corollary is somewhat more general than the Neumann-Praeger result stated above, in the sense that it also deals with $A = M_n(q^m)$
as a $\F_q$-algebra, but it is obtained in \cite{KMP} using different methods.

An equivalent formulation of Corollary \ref{randomgenerationcor} is as follows. Let $A$ be a simple algebra and consider the free associative algebra $k\langle X_1,X_2\rangle$. Then the probability that a randomly chosen $k$-algebra homomorphism $k\langle X_1,X_2\rangle \to A$ is surjective tends to $1$ as $|A| \to \infty$.

It is well known that any finite simple algebra is $2$-generated, see for instance Theorem $6.4$ of \cite{KMP}. So it follows from Corollary \ref{randomgenerationcor} that there exists an absolute constant $\delta>0$ such that $P(A) \geq \delta$ for all finite simple algebras $A$. In the following result, we find the best possible value for this constant.

\begin{thm}\label{minP} Let $A$ be a finite simple algebra. Then $P(A) \geq 3/8$, with equality if and only if $A=M_2(2)$.
\end{thm}

For $G$ a finite simple group, let $P(G)$ be the probability that two randomly chosen elements of $G$ will generate $G$. It is a consequence of Theorem $1.1$ of \cite{MQR} that $P(G) \geq 53/90$, with equality if and only if $G=A_6$.

For $A$ simple and not a field, we investigate the growth rate of $P(A)$ in more detail. Let $m(A)$ be the minimal index (as an additive group) of any proper subalgebra of $A$.

\begin{thm}\label{estimateprob} Let $A$ be a finite simple algebra that is not a field. Then $$P(A)=1-\kappa(A)m(A)^{-1}+ O(m(A)^{-4/3})$$ where $\kappa:A \to \R$ is a function satisfying $1<\kappa(A) < 4$.
\end{thm}

We will see in Section \ref{proofestimateprob} that the constants in Theorem \ref{estimateprob} are best possible. Note that Theorem \ref{estimateprob} gives us an alternate proof of Corollary \ref{randomgenerationcor}. Results of this flavour for finite simple groups
were obtained by Liebeck and Shalev, see Theorems 1.5 and 1.6 in \cite{LiSh}.

We next look at randomly generating a finite algebra by its nilpotent elements.

Define $P_N(A)$ to be the probability that two nilpotent elements of $A$ chosen uniformly at random will generate $A$ as a $k$-algebra. That
is, $$P_N(A)=\frac{|\{(x,y) \in A^N \times A^N: \langle x,y\rangle=A\}|}{|A^N|^2}.$$ We prove an analogue of Theorem \ref{randomgeneration}.

\begin{thm}\label{nilpotentrandomgeneration} Fix constants $1<c <q^{1/4}$ and $\lambda >0$. Let $A$ be a finite algebra, say $A = \big(\prod_{i=1}^r M_{n_i}(q^{m_i})\big) \oplus J(A)$, that is bounded by $(c,\lambda)$. Denote $n:=\min_{i=1,...,r}\{n_i\}$ and $m:=\min_{i=1,...,r}\{m_i\}$. Assume that $n>1$. Then $P_N(A) \to 1$ as $n \to \infty$, as $m \to \infty$ or as $q \to \infty$.
\end{thm}

Note that Theorem \ref{nilpotentrandomgeneration} does not hold when $n=1$. For example, let $A_0$ be a finite algebra and let $A=\F_{q^m} \times A_0$ for some $m>1$. Let $b$ be a prime divisor of $m$ and consider the maximal subalgebra $B=\F_{q^{m/b}} \times A_0$ of $A$. Observe that all nilpotent elements of $A$ are contained in $B$. So $P_N(A)=0$, regardless of the choice of $q$, $m$ or $A_0$.

Theorem \ref{nilpotentrandomgeneration} immediately implies the following.

\begin{cor}\label{nilpotentrandomgenerationcor} Let $A$ be a finite simple algebra that is not a field. Then $P_N(A) \to 1$ as $|A| \to \infty$.
\end{cor}

We now consider random generation of a finite simple algebra by two elements that have a given characteristic polynomial. Let $A=M_n(q^m)$, let $f$ be a monic polynomial of degree $n$ over $\F_{q^m}$ and let $A_f$ be the set of elements of $A$ with characteristic polynomial $f$. We define $P_f(A)$ to be the probability that two elements of $A_f$ chosen uniformly at random will generate $A$ as a $k$-algebra. That is, $$P_f(A)=\frac{|\{(x,y) \in A_f \times A_f: \langle x,y\rangle=A\}|}{|A_f|^2}.$$


\begin{thm}\label{polyrandomgeneration} Let $A$ be a finite simple algebra that is not a field, say $A=M_n(q^m)$ for $n>1$. Let $f$ be a monic polynomial of degree $n$ over $\F_{q^m}$. Then $P_f(A) \to 1$ as $|A| \to \infty$.
\end{thm}

By applying Theorem \ref{polyrandomgeneration} to the case where $f(X)=X^n$, we find an alternate proof of Corollary \ref{nilpotentrandomgenerationcor}.

Note that Theorem \ref{polyrandomgeneration} does not hold when $A$ is a field. For example, let $A=\F_{q^m}$ for some $m>1$. Let $b$ be a prime divisor of $m$ and consider the maximal subfield $B=\F_{q^{m/b}}$ of $A$. Let $x \in B$ and let $f$ be the polynomial $X-x$ over $\F_{q^m}$. Then $A_f=B_f=\{x\}$, and so $P_f(A)=0$ regardless of the choice of $q$ or $m$.

We remark that Theorem \ref{polyrandomgeneration} still holds, with essentially the same proof, if we replace $P_f(A)$ with $P_{f,g}(A)$, where $g$ is another monic polynomial of degree $n$ over $\F_{q^m}$ and $P_{f,g}(A)$ is the probability that a random element of $A_f \times A_g$ will generate $A$ as a $k$-algebra.

We now consider random generation of a finite simple algebra by two matrices that have a given rank. Let $\alpha$ be a non-negative integer. Let $A=M_n(q^m)$ where $n \ge \alpha$ and let $A_{\alpha}$ be the set of matrices in $A$ with rank $\alpha$. We define $P_{\alpha}(A)$ to be the probability that two elements of $A_{\alpha}$ chosen uniformly at random will generate $A$ as a $k$-algebra. That is, $$P_{\alpha}(A)=\frac{|\{(x,y) \in A_{\alpha} \times A_{\alpha}: \langle x,y\rangle=A\}|}{|A_{\alpha}|^2}.$$

\begin{thm}\label{rankrandomgeneration}
Let $A$ be a finite simple algebra that is not a field, say $A=M_n(q^m)$ for $n>1$. Let $\alpha:=\alpha(n)$ be a positive integer.

$(i)$ Let $p$ be the smallest prime divisor of $n$. If $\alpha \leq n/p$ then $P_{\alpha}(A) \leq 1- q^{-2mp\alpha^2}$.

$(ii)$ If $n - \sqrt{n}/3 \le \alpha \leq n$ then $P_{\alpha}(A) \to 1$ as $|A| \to \infty$.
\end{thm}

It is not true that $P_{\alpha}(A)$ always tends to $1$ as $|A| \to \infty$. This is an immediate consequence of Theorem \ref{rankrandomgeneration}$(i)$. We can see this by taking $\alpha$ to be independent of $n$, and letting $n$ tend to infinity whilst fixing $q$, $m$ and $p$.

Let $P^{\times}(A)$ to be the probability that two invertible elements of $A$ chosen uniformly at random will generate $A$ as a $k$-algebra.
Theorem \ref{rankrandomgeneration}$(ii)$ implies the following.

\begin{cor}\label{rankrandomgenerationcor} Let $A$ be a finite simple algebra. Then $P^{\times}(A) \to 1$ as $|A| \to \infty$.
\end{cor}

Next, we investigate the minimal number of generators $d(A)$ of a finite algebra $A$.

An obvious upper bound for $d(A)$ is $\log_q|A| - 1$ (the $-1$ term arises from our convention that the multiplicative identity of $A$ is automatically included in any generating set of $A$, and of course $\smash{\dim A = \log_q|A|}$). This upper bound is strict, and is realised in the case where $J(A)$ has codimension $1$ in $A$ and $J(A)^2=0$. 

In general $d(A)$ often grows much slower than $\log_q|A|$. For example, if $A$ is the direct product of finitely many copies of $k$ then $d(A) = \smash{\left \lceil{\log_q \log_q |A|}\right \rceil}$. In particular, if $A = k^r$ for some $1 < r \leq q$ then $d(A)=1$. Moreover, as remarked earlier, if $A$ is simple then $d(A)=2$.

\begin{thm}\label{mind} Let $A$ be a finite algebra, say $A=S \oplus J(A)$ where $S=\prod_{i=1}^rS_i^{\alpha_i}$, $S_i=M_{n_i}(q^{m_i})$ for each $i$ and the $S_i$'s are pairwise non-isomorphic. Let $f(A,i):=m_i^{-1}n_i^{-2}\log_q \alpha_i m_i$, let $f(A):=\max_i\{f(A,i)\}$ and let $\mu(A)$ be the minimal length of an unrefinable chain of $S$-subbimodules of $J(A)$. Then $$-2.33 < d(A) -f(A) < \mu(A) + 3.42.$$
\end{thm}

In the final part of this paper we study positively finitely generated (profinite) algebras and related topics.
For the theory of positively finitely generated groups see \cite{MSh, BPSh, P, Po, DLM, Lu, JP} and the references therein.

A \textit{profinite algebra} is a topological algebra (over $k$) that is isomorphic to a projective limit of discrete finite algebras.
Henceforth, let $A$ be a profinite algebra.

For $d \ge 1$ let $P(A,d)$ be the probability that $d$ randomly chosen elements of $A$ generate $A$
(topologically if $A$ is infinite).  We say that $A$ is \emph{positively finitely generated} (PFG)
if $P(A,d) > 0$ for some $d$. We say that $A$ has \emph{polynomial maximal subalgebra growth}
(PMSG) if the number $m_n(A)$ of index $n$ (open) maximal subalgebras of $A$ is bounded by $n^c$
for some fixed constant $c$. It was shown in \cite{MSh} that, for profinite groups, PFG is equivalent to PMSG.
Here we study these notions and related invariants for profinite algebras.

If we do not specify a base, $\log$ refers to base $2$. Set
\[
M(A) := \sup_{n > 1} \log{m_n(A)}/ \log{n}, \; \;  M^*(A) := \limsup_{n > 1} \log{m_n(A)}/ \log{n},
\]
which measure the degree of polynomial subgroup growth of $A$ (and are infinite unless $A$ has PMSG). Let $d_0(A):=\min\{d \ge 1 \hspace{0.5mm}|\hspace{0.5mm} P(A,d) > 0\}$.

We establish the following.

\begin{thm}\label{PFG} Let $A$ be a profinite algebra. Then $A$ is PFG if and only
if $A$ has PMSG. Moreover, if $A$ is infinite we have $M^*(A) \le d_0(A)+1$.
\end{thm}

The bound above is better than related bounds obtained for profinite groups.

For any real number $\eta \geq 1$, define the Pomerance invariant of $A$ by $$V_{\eta}(A) := \min \{ d \ge 1: P(A,d) > \eta ^{-1} \}.$$ Clearly $V_{\eta}(A) \geq d_0(A)$, with equality for sufficiently large $\eta$. The case where $\eta = e$, which we denote by $V(A):=V_{e}(A)$, was studied by Pomerance \cite{Po} for finite abelian groups.

Next, define the Pak invariant $E(A)$ of $A$ to be the expected number of random elements of $A$ chosen uniformly and independently which generate $A$ (topologically). A similar invariant was introduced by Pak \cite{P} for finite groups.

Our final main result establishes bounds on these invariants, and is a ring-theoretic analogue of results of Lubotzky \cite{Lu} and Lucchini-Moscatiello \cite{LM} for finite groups.

\begin{thm}\label{PSAA} Let $A$ be a finite algebra, say $A/J(A) = \prod_{i=1}^rS_i$. Then

$(i)$ $M(A) \le  2\log_q r + d(A)+2$.

$(ii)$ $\lceil M(A)- 5.24\rceil  \le V(A) \le \lceil M(A) +2.02 \rceil$.

$(iii)$ $\lceil M(A) -5.80 \rceil \le E(A) \le \lceil M(A) \rceil + 3$.
\end{thm}

In particular, the expected number of random elements of $A$ which generate $A$
is of the order of magnitude $O(d(A) + \log_q \log_q |A|)$.

This paper is structured as follows. In Section $\ref{preliminaries}$ we present a classification of maximal subalgebras of a finite algebra $A$, then we introduce and investigate a related zeta function of $A$. In Sections \ref{proofrandomgeneration}, \ref{proofminP} and \ref{proofestimateprob} we investigate $P(A)$ and its growth rate. In particular, in Section \ref{proofrandomgeneration} we prove Theorem \ref{randomgeneration} and Corollary \ref{randomgenerationcor}, in Section \ref{proofminP} we prove Theorem \ref{minP} and in Section \ref{proofestimateprob} we prove Theorem \ref{estimateprob}. In Sections \ref{proofnilpotentrandomgeneration}, \ref{proofpolyrandomgeneration} and \ref{proofrankrandomgeneration} we study random generation of a finite algebra by special elements. In Section \ref{proofnilpotentrandomgeneration} we prove Theorem \ref{nilpotentrandomgeneration} and Corollary \ref{nilpotentrandomgenerationcor}, in Section \ref{proofpolyrandomgeneration} we prove Theorem \ref{polyrandomgeneration} and in Section \ref{proofrankrandomgeneration} we prove Theorem \ref{rankrandomgeneration} and Corollary \ref{rankrandomgenerationcor}. In Section \ref{EstimatingdA} we look at the minimal number of generators of a finite algebra, and prove Theorem \ref{mind}. Finally, in Section \ref{PFGA} we investigate positively finitely generated profinite algebras, and prove Theorems \ref{PFG} and \ref{PSAA}.

\section{Preliminaries}\label{preliminaries}

Recall that $k=\F_q$ where $q$ is a prime power.

Let $A$ be an (associative, unital) finite simple algebra (over $k$). By Wedderburn's Theorem, we can write $A=M_n(q^m)$ for some positive integers $n$ and $m$.

Some remarks on notation. Let $\alpha=(\alpha_1,...,\alpha_s)$ be a composition of $n$ (i.e. $n=\sum_{i=1}^s\alpha_i$ where the $\alpha_i$'s
are positive integers) and suppose $s \geq 2$. Let $P_{\alpha}(q^m)$ be the subalgebra of $A$ that consists of all block upper triangular
matrices with $s$ blocks on the diagonal such that the $i$'th block has size $\alpha_i$.

Let $r$ be a positive integer. There is a natural embedding of $\F_{q^r}$ in $M_r(q)$ via the left regular representation.  If $r$ divides $n$
then this extends to an embedding of $M_{n/r}(q^{mr})$ in $M_n(q^m)$. If $r$ divides $m$ then the subfield $\F_{q^{m/r}}$ of $\F_{q^m}$
extends naturally to a subalgebra $\smash{M_n(q^{m/r})}$ of $M_n(q^m)$. Let $\mathcal{P}(r)$ denote the set of prime divisors of $r$ (not counting multiplicities). Let $\omega(r):=|\mathcal{P}(r)|$.

We define three sets of subalgebras of $A$;

\noindent $S1:=\{ P_{l,n-l}(q^m) \hspace{0.5mm}|\hspace{0.5mm} l \in \N, l<n\}$,

\vspace{0.5mm}\noindent $S2:=\{M_{n/a}(q^{ma}) \hspace{0.5mm}|\hspace{0.5mm} a \in \mathcal{P}(n)\}$, and

\noindent $S3:=\{M_n(q^{m/b}) \hspace{0.5mm}|\hspace{0.5mm} b \in \mathcal{P}(m)\}$.

A subalgebra of $A$ that is conjugate to an element of $S1$ (resp. $S2$, $S3$) is said to be \textit{of type $(S1)$} (resp. $(S2)$, $(S3)$).

\begin{theorem}\label{classthm1} Let $A$ be a finite simple algebra. With the above notation, $S1 \cup S2 \cup S3$ is a set of representatives of the conjugacy classes of maximal subalgebras of $A$
\end{theorem}

\begin{proof} Over any field $k$, Lemma $3.6$ of Iovanov and Sistko \cite{IS} classifies maximal subalgebras of a simple $k$-algebra up to
isomorphism. We adapt this result to the case where $k=\F_q$, and then we consider conjugacy classes.

Let $B$ be a maximal subalgebra of $A$. If $B$ is not simple then, by Lemma $3.6$ of \cite{IS}, $B$ is conjugate to $P_{l,n-l}(q^m)$ for some
positive integer $l<n$. Let $l'<n$ be a positive integer. It is well known that $P_{l,n-l}(q^m)$ is conjugate to $P_{l',n-l'}(q^m)$ if and
only if $l=l'$ (see for instance $\S 3$ of \cite{EK}).

Henceforth let $B$ be simple. By Lemma $3.6$ of \cite{IS}, there are two possibilities. Either $Z(B) \supseteq Z(A)$ or $Z(A) \supseteq Z(B)$.

Assume that $Z(B) \supseteq Z(A)$. Then, by Lemma $3.6$ of \cite{IS}, $B=C_A(F)$ for some minimal field extension $F$ of $Z(A)$ that is
contained in $A$. Observe that $Z(A)  \cong \F_{q^m}$. So $F \cong \F_{q^{ma}}$ for some prime divisor $a$ of $n$. By the double centraliser
theorem (Theorem $7.1.9$ of \cite{R}), $Z(B) = F$ and $[F:Z(A)][B:Z(A)]=[A:Z(A)]$. Recall from Wedderburn’s little theorem that all finite
division algebras are fields. It follows that $B \cong M_{n/a}(F)$. Any subalgebra of $A$ that is isomorphic to $B$ is then conjugate to $B$ by the Skolem-Noether theorem.

Now assume that $Z(A) \supseteq Z(B)$. Then, by Lemma $3.6$ of \cite{IS}, $Z(B)$ is a maximal subfield of $Z(A)$ that contains $k$ such that
$A \cong Z(A) \otimes_{Z(B)} B$. So $Z(B) \cong \F_{q^{m/b}}$ for some prime divisor $b$ of $m$. Since $A$ and $B$ are both simple, it follows
from Wedderburn's theorem that $B \cong M_n(q^{m/b})$.

Let $\iota:B \hookrightarrow A$ be inclusion. Observe that $\iota$ extends to a $Z(A)$-isomorphism $\iota^*:B \otimes_{Z(B)} Z(A) \to A$. Let $B'$ be another subalgebra of $A$ and let $f:B \to B'$ be a $k$-isomorphism. Let $\iota':B' \hookrightarrow A$ be inclusion and denote $\tau:=\iota' \circ f$. Then $\tau$ extends to a $Z(A)$-isomorphism $\tau^*:B \otimes_{Z(B)} Z(A) \to A$. By the Skolem-Noether theorem, there exists $g \in A^{\times}$ such that $g\tau^*(x)g^{-1}=\iota^*(x)$ for all $x \in B \otimes_{Z(B)} Z(A)$. Hence $B'$ is conjugate to $B$. This completes the proof.
\end{proof}

We call $S1 \cup S2 \cup S3$ the \textit{standard} set of representatives of the conjugacy classes of maximal subalgebras of $A$.

We now relax the assumption that $A$ is simple. Let $A$ be any finite algebra over $k$. By the Wedderburn-Malcev Principal Theorem, there exists a semisimple subalgebra $S$ of $A$ such that $A=S \oplus J(A)$. Decompose $S=\prod_{i=1}^rS_i$ where each $S_i$ is simple. Let $i \in \{1,...r\}$. Write $S_i=M_{n_i}(q^{m_i})$ for some integers $m_i$ and $n_i$. Let $\mathcal{B}_i$ be the standard set of representatives of the conjugacy classes of maximal subalgebras of $S_i$. If $S_j \cong S_i$ for some $j \neq i$ then let $S_{ij}$ denote the image of the diagonal
embedding $S_i \to S_i \times S_j$.

We define three sets of subalgebras of $A$;

\noindent $T1:=\big\{(B_j \times \prod_{i \neq j} S_i) \oplus J(A) \hspace{0.5mm}\big|\hspace{0.5mm} 1 \leq j \leq r; B_j \in
\mathcal{B}_j\big\}$,

\vspace{0.5mm}\noindent $T2:=\big\{(S_{j_1j_2} \times \prod_{i \neq j_1, j_2} S_i) \oplus J(A) \hspace{0.5mm}\big|\hspace{0.5mm} 1 \leq j_1 <
j_2 \leq r, S_{j_1} \cong S_{j_2}\big\}$, and

\vspace{0.5mm}\noindent $T3:=\big\{S \oplus H \hspace{0.5mm}\big|\hspace{0.5mm} \textnormal{$H$ is a two-sided ideal of $A$ that is maximal
with respect to $H \subset J(A)$}\big\}$.

A subalgebra of $A$ that is conjugate to an element of $T1$ (resp. $T2$, $T3$) is said to be \textit{of type $(T1)$} (resp. $(T2)$, $(T3)$).

\begin{theorem}\label{classthm2} Let $A$ be a finite algebra. With the above notation, $T1 \cup T2 \cup T3$ is a set of representatives of the conjugacy classes of maximal subalgebras of $A$.
\end{theorem}

\begin{proof} By Theorems $2.5$ and $3.10$ of \cite{IS}, every maximal subalgebra of $A$ is conjugate to an element of $T1 \cup T2 \cup T3$.
It remains to check that all elements of $T1 \cup T2 \cup T3$ are pairwise non-conjugate in $A$.

We first consider the case where $A$ is semisimple, that is, $J(A)=0$. Note that $T3=\varnothing$. It is easy to see that the elements of $T1
\cup T2$ are pairwise non-conjugate as the simple components of $A$ commute with each other.

We now consider the general case. That is, $A$ is any algebra. Let $B, B' \in T1 \cup T2 \cup T3$ and let $a \in A^{\times}$ such that
$B^a := a^{-1}Ba = B'$. Write $a=s+j$ for $s \in S$ and $j \in J(A)$.

Assume that $B,B' \in T1 \cup T2$. Write $B=M \oplus J(A)$ and $B'=M'\oplus J(A)$. Observe that $M^s=M'$ since $J(A)$ is a two-sided ideal of
$A$. Hence $M=M'$ as $S$ is semisimple.

Next assume that $B,B' \in T3$. Write $B=S \oplus H$ and $B'=S \oplus H'$. Then $H^a=H=H'$ since $H$ and $H'$ are two-sided ideals of $A$.

Finally, if $B \in T3$ and $B' \in T1 \cup T2$ (or vice versa) then $B \not\cong B'$, a contradiction.
\end{proof}

We call $T1 \cup T2 \cup T3$ the \textit{standard} set of representatives of the conjugacy classes of maximal subalgebras of $A$.

We now introduce a \lq zeta function' of $A$. Let $\mathcal{B}$ be the standard set of representatives of the conjugacy classes of maximal
subalgebras of $A$. For $\epsilon > 0$, we define \begin{equation}\label{zeta}\zeta_A(\epsilon) = \sum_{B \in
\mathcal{B}}(|A|/|B|)^{-\epsilon}\end{equation} where $\zeta_A(\epsilon)=0$ if $A=k$. Next, we prove a result which serves as a main tool in this paper. Recall the notation $A= \big(\prod_{i=1}^r M_{n_i}(q^{m_i})\big) \oplus J(A)$. Denote $n:=\min_{i=1,...,r}\{n_i\}$ and $m:=\min_{i=1,...,r}\{m_i\}$.

\begin{theorem}\label{zetatheorem} Fix constants $\lambda >0$ and $\epsilon>0$. With the above notation, there exists $c=c(\epsilon)>1$ such that if $A$ is a finite algebra that is bounded by $(c,\lambda)$ then $\zeta_A(\epsilon) \to 0$ as $n \to \infty$, as $m \to \infty$ or as $q \to \infty$.
\end{theorem}
\begin{proof} Fix $\epsilon >0$. Let $\mathcal{B}$ be the standard set of representatives of the conjugacy classes of maximal subalgebras of $A$. Let $B \in
\mathcal{B}$.

We first consider the case where $A$ is simple. That is, $A=M_n(q^m)$. Let $\Sigma_1$ (resp. $\Sigma_2$, $\Sigma_3$) denote the contribution to the sum in $(\ref{zeta})$ of the maximal subalgebras in $S1$ (resp. $S2$, $S3$).

We consider individually each of the possibilities that $B$ is in $S1$, $S2$ or $S3$.

Let $B \in S1$. That is, $B=P_{l,n-l}(q^m)$ for some positive integer $l <n$. Observe that $|B|=q^{m(n^2-l(n-l))}$. Then
$$\Sigma_1=\sum_{l=1}^{n-1}q^{-\epsilon ml(n-l)} \leq (n-1)q^{-\epsilon m(n-1)}.$$

Let $B \in S2$. That is, $B=M_{n/a}(q^{ma})$ for some prime divisor $a$ of $n$. Observe that $|B|=q^{mn^2/a}$. Then $$\Sigma_2=\sum_{a \in
\mathcal{P}(n)}q^{-\epsilon mn^2(1-1/a)} \leq \omega(n)q^{-\epsilon mn^2/2}.$$

Let $B \in S3$. That is, $B=M_n(q^{m/b})$ for some prime divisor $b$ of $m$. Observe that $|B|=q^{mn^2/b}$. Then $$\Sigma_3=\sum_{b \in
\mathcal{P}(m)}q^{-\epsilon mn^2(1-1/b)} \leq \omega(m)q^{-\epsilon mn^2/2}.$$

Observe that, since $\omega(n) \le n-1$, we have
$$\zeta_A(\epsilon)=\Sigma_1+\Sigma_2+\Sigma_3 \leq (2(n-1)+\omega(m))q^{-\epsilon mn/2}.$$
So $\zeta_A(\epsilon) \to 0$ as $n \to \infty$, as $m \to \infty$ or as $q \to \infty$.

This completes the proof for the case where $A$ is simple.

We now consider the general case. That is, $A=S \oplus J(A)$ where $S=\prod_{i=1}^rS_i$ is semisimple and $S_i=M_{n_i}(q^{m_i})$ for each $i$. Let $\Omega_1$ (resp. $\Omega_2$, $\Omega_3$) denote the contribution to the sum in $(\ref{zeta})$ of the maximal subalgebras in $T1$ (resp. $T2$, $T3$).

Let $i_0 \in \{1,...,r\}$ satisfy $\zeta_{S_{i_0}}(\epsilon) \geq \zeta_{S_i}(\epsilon)$ for all $1 \leq i \leq m$. For simplicity, denote
$n_0:=n_{i_0}$ and $m_0:=m_{i_0}$.

Let $c \in \R$ such that $1<c<q^{\epsilon}$ and let $\lambda >0$. We impose the condition that $A$ is bounded by $(c,\lambda)$. That is, $r \leq \lambda c^{\min\{m,n\}/2}$ and $\dim J(A)/J(A)^2 \leq \log_q\lambda +{\min\{m,n\}^2}\log_q c$. Rearranging this second inequality gives us $|J(A)/J(A)^2| \leq \lambda \smash{c^{\min\{m,n\}^2}}$.

Let $B \in T1$. That is, $B = (B_j \times \prod_{i \neq j} S_i) \oplus J(A)$ for some $j \in \{1,...,r\}$ and maximal subalgebra $B_j$ of $S_j$. Then we
have $$\Omega_1=\sum_{j=1}^r \zeta_{S_j}(\epsilon) \leq r\zeta_{S_{i_0}}(\epsilon) \leq r(2(n_0-1)+\omega(m_0))q^{-\epsilon m_0n_0/2}.$$
So $\Omega_1 \to 0$ as $n \to \infty$, as $m \to \infty$ or as $q \to \infty$.

Let $B \in T2$. That is, $B =(S_{j_1j_2} \times \prod_{i \neq j_1, j_2} S_i) \oplus J(A)$ for some $1 \leq j_1 < j_2 \leq r$ such that $S_{j_1} \cong
S_{j_2}$. Observe that $|A|/|B|=|S_{j_1}| \geq q^{mn^2}$. Then $$\Omega_2 \leq \sum_{1 \leq j_1 < j_2 \leq r} (q^{mn^2})^{-\epsilon} = {r
\choose 2}q^{-\epsilon mn^2}.$$ So $\Omega_2 \to 0$ as $n \to \infty$, as $m \to \infty$ or as $q \to \infty$.

Finally, let $B \in T3$. That is, $B=S \oplus H$ where $H$ is a two-sided ideal of $A$ that is maximal with respect to the condition $H \subset J(A)$.

Let $S^{op}$ denote the opposite algebra of $S$. Observe that $J(A)/H$ is a non-trivial simple $S$-bimodule and hence, by the equivalence of categories in Proposition $10.1$ of \cite{Pi}, $J(A)/H$ also has the structure of a non-trivial simple left $S \otimes_k S^{op}$-module. Consider the $k$-algebra isomorphism $S \otimes_k S^{op} \cong \prod_{1 \leq i,j \leq r} M_{n_in_j}(q^{m_im_j})$. 
Then, by Proposition $2.3$ of \cite{Pi}, any simple left module of $S \otimes_k S^{op}$ is isomorphic to $(\F_{q^{m_im_j}})^{n_in_j}$ for some $i,j \in \{1,...,r\}$. Hence $\smash{|A|/|B| =|J(A)/H| \geq q^{m^2n^2}}$.

Let $\mathcal{H}$ be the set of two-sided ideals of $A$ that are maximal with respect to being properly contained in $J(A)$. By the proof of Theorem $2.5$ of \cite{IS}, all ideals in $\mathcal{H}$ contain $J(A)^2$. So certainly $|\mathcal{H}| \leq |J(A)/J(A)^2|$. 
Hence $$\Omega_3 \leq q^{-\epsilon m^2n^2} |J(A)/J(A)^2| \leq \lambda q^{m^2n^2 (-\epsilon + \log_q c) }.$$

So $\Omega_3 \to 0$ as $n \to \infty$, as $m \to \infty$ or as $q \to \infty$. This completes the proof.
\end{proof}

\begin{corollary}\label{zetatheoremcor} Let $\epsilon >0$ and let $A$ be a finite simple algebra. Then $\zeta_A(\epsilon) \to 0$ as $|A| \to \infty$.
\begin{proof} Write $A=M_n(q^m)$. Recall from the proof of Theorem \ref{zetatheorem} that $\zeta_A(\epsilon) \to 0$ as $n \to \infty$, as $m \to \infty$ or as $q \to \infty$. The result follows immediately as $|A|=q^{mn^2}$.
\end{proof}
\end{corollary}

\begin{lemma}\label{deconstruct} Let $A$ be a finite algebra and let $S$ be a semisimple subalgebra of $A$ such that $A=S \oplus J(A)$. Then $A^{\times}=S^{\times} \times J(A)$ and $A^N=S^N \times J(A)$, where $\times$ denotes Cartesian product of sets.
\begin{proof} Let $g \in A^{\times}$. Write $g=s+j$ and $g^{-1}=s'+j'$ for $s,s' \in S$ and $j,j' \in J(A)$. Then $1=gg^{-1}=ss'+sj'+js'+jj'$,
where $sj'+js'+jj' \in J(A)$. Hence $s'=s^{-1}$. Conversely, let $a=s_0+j_0 \in S^{\times} \times J(A)$. Observe that
$s_0^{-1}-j_0s_0^{-1}/(s_0+j_0)=a^{-1}$.

Let $x \in A^N$ and let $\alpha$ be the (nilpotency) index of $x$. Write $x=s_1+j_1$ for $s_1 \in S$ and $j_1 \in J(A)$. Then
$0=x^\alpha=s_1^{\alpha}+j_1'$, for some $j_1' \in J(A)$. Hence $s_1 \in S^N$. Conversely, let $y=s_2+j_2 \in S^N \times J(A)$ and let
$\beta$ be the index of $s_2$. Then $y^\beta \in J(A)$ and so $y \in A^N$.
\end{proof}
\end{lemma}

For positive integers $u,v$, define a function $$F(u,v)=(1-u^{-1})(1-u^{-2})...(1-u^{-v}) $$ where $F(u,0)=1$. We will need the following elementary lemmas.

\begin{lemma}\label{babylemma} Let $u,v,c \in \N$. Then $F(u,v)^c \leq F(u^c,v) \leq 2^vF(u,v)$.
\begin{proof} If $u=1$ then $F(u,v)=0$ and the inequality holds. So assume that $u>1$.

Observe that $u^c-(u-1)^c \geq 1$. Rearranging, we have $1-u^{-c} \geq (1-u^{-1})^c$. The lower bound then follows immediately since $u$ is arbitrary.

For the upper bound, observe that $(1-u^{-c}) \leq 2(1-u^{-1})$. Then we are done again since $u$ is arbitrary.
\end{proof}
\end{lemma}

\begin{lemma}\label{babylemma1} Let $u,v,w \in \N$ such that $w < v$. Then $F(u,v) \leq \big(\frac{3}{2}\big)^{v/2}F(u,w)F(u,v-w)$.
\begin{proof} If $u=1$ then we are done. So assume that $u>1$. Let $x \in \N$. We first show that \begin{equation}\label{littlethingy} \frac{(1-u^{-(x+1)})...(1-u^{-2x})}{(1-u^{-1})...(1-u^{-x})} \leq \Big(\frac{3}{2}\Big)^{x}\end{equation} by induction on $x$. If $x=1$ then it certainly holds. If $x>1$ then $$\frac{(1-u^{-(x+1)})...(1-u^{-2x})}{(1-u^{-1})...(1-u^{-x})} \leq \Big(\frac{3}{2}\Big)^{x-1}\frac{(1-u^{-2x})}{(1-u^{-x})} \leq \Big(\frac{3}{2}\Big)^{x}$$ using the inductive hypothesis.

Without loss of generality, assume that $w \leq v/2$ (otherwise we swap $w$ and $v-w$). Using $(\ref{littlethingy})$, we have \begin{align*}\frac{F(u,v)}{F(u,w)F(u,v-w)}& =\frac{(1-u^{-(v-w+1)})...(1-u^{-v})}{(1-u^{-1})...(1-u^{-w})} \\ & \leq \frac{(1-u^{-({\lfloor v/2 \rfloor}+1)})...(1-u^{-2\lfloor v/2 \rfloor})}{(1-u^{-1})...(1-u^{-\lfloor v/2 \rfloor})} \\ & \leq \Big(\frac{3}{2}\Big)^{v/2}. \qedhere \end{align*}
\end{proof}
\end{lemma}

One can use Leibniz's alternating series test to show that $F(u,v)$ converges towards a positive limit as $v \to \infty$ and $u$ is fixed. This limit is also known as $\phi(1/u)$, where $\phi$ denotes the Euler function. It is known that $\phi(1/u)$ is transcendental. For example, $\phi(1/2) \approx 0.2888$.

\begin{lemma}\label{babylemma0} Let $A$ be a finite simple algebra, say $A=M_n(q^m)$. Then $$\phi(1/2)<\frac{|A^{\times}|}{|A|}=F(q^m,n) <1.$$ 
\begin{proof} It is easy to check that $|A^{\times}|/|A|=q^{-mn^2}\prod_{i=0}^{n-1}(q^{mn}-q^{mi})=F(q^m,n)$. Observe that $F(q^m,n)$ is monatonically decreasing (resp. increasing) in $n$ (resp. $q^m$). Then the result follows from the remark preceding this lemma.
\end{proof}
\end{lemma}

Finally, we will need the elementary inequality \begin{equation}\label{elementary} x/y \leq (x-1)/(y-1) \leq 2x/y \end{equation} for all integers $x\geq y \geq 2$. 

\section{Proof of Theorem \ref{randomgeneration} and Corollary \ref{randomgenerationcor}}\label{proofrandomgeneration}

Let $A$ be a finite algebra, say $A=S \oplus J(A)$ where $S=\prod_{i=1}^rS_i$ is semisimple and $S_i=M_{n_i}(q^{m_i})$ for each $i$. Denote $n:=\min_{i=1,...,r}\{n_i\}$ and $m:=\min_{i=1,...,r}\{m_i\}$. Recall that $\phi(1/2) \approx 0.2888$.

We begin by considering two examples. Let $p$ be a prime. Let $A=M_p(2^p)$ and let $B=M_p(2)$. Then $\frac{|A^{\times}|}{|B^{\times}|}\frac{|B|}{|A|} = \frac{F(2^p,p)}{F(2,p)} \to \phi(1/2)^{-1}$ as $p \to \infty$. Now let $A=M_p(2)$ and let $B=\F_{2^p}$. Then $\frac{|A^{\times}|}{|B^{\times}|}\frac{|B|}{|A|} = \frac{F(2,p)}{F(2^p,1)} \to \phi(1/2)$ as $p \to \infty$. So we see that the constants in the following lemma are best possible.

\begin{lemma}\label{unitslemma} Let $B$ be a maximal subalgebra of $A$. Then $\phi(1/2) < \frac{|A^{\times}|}{|B^{\times}|}\frac{|B|}{|A|} < \phi(1/2)^{-1}$.
\begin{proof} We first consider the case where $A$ is simple. That is, $A=M_n(q^m)$. Note that $A^{\times}=\GL_n(q^m)$. For simplicity, denote $t:=q^m$.

Assume that $B$ is of type $(S1)$. That is, $B \cong P_{l,n-l}(t)$ for some positive integer $l <n$. Observe that $|B^{\times}|=|(B/J(B))^{\times}||J(B)|$ by Lemma \ref{deconstruct}. Then applying Lemma \ref{babylemma0} gives us $$\frac{|A^{\times}|}{|B^{\times}|} \frac{|B|}{|A|}=\frac{|A^{\times}|}{|A|}\frac{|M_l(t)|}{|M_l(t)^{\times}|}\frac{|M_{n-l}(t)|}{|M_{n-l}(t)^{\times}|}=\frac{F(t,n)}{F(t,l)F(t,n-l)}$$ and hence $$\phi(1/2)<\frac{|A^{\times}|}{|B^{\times}|} \frac{|B|}{|A|}<\frac{1}{F(t,l)} <\phi(1/2)^{-1}.$$

Now assume that $B$ is not of type $(S1)$. Then $B$ is simple by Theorem \ref{classthm1}. Hence, by Lemma \ref{babylemma0}, we have $$\phi(1/2) < \frac{|A^{\times}|}{|B^{\times}|} \frac{|B|}{|A|} < \phi(1/2)^{-1}.$$

This completes the proof for the case where $A$ is simple.

We now consider the general case. Recall that $A=S \oplus J(A)$ where $S=\prod_{i=1}^rS_i$ is semisimple. By Theorem \ref{classthm2}, $B$ is of type $(T1)$, $(T2)$ or $(T3)$. We consider each of these possibilities.

Let $B$ be of type $(T1)$. That is, $B \cong (B_j \times \prod_{i \neq j} S_i) \oplus J(A)$ for some $j \in \{1,...,r\}$ and maximal subalgebra $B_j$ of $S_j$. Applying Lemma \ref{deconstruct} gives us $$\frac{|A^{\times}|}{|B^{\times}|} = \frac{\prod_{i=1}^r|S_i^{\times}|}{|B_j^{\times}|\cdot \prod_{i\neq j}|S_i^{\times}|} = \frac{|S_j^{\times}|}{|B_j^{\times}|}.$$ Since $S_j$ is simple, Lemma \ref{babylemma0} gives us $$\phi(1/2)\frac{|A|}{|B|}= \phi(1/2)\frac{|S_j|}{|B_j|} < \frac{|A^{\times}|}{|B^{\times}|} < \phi(1/2)^{-1}\frac{|S_j|}{|B_j|}=\phi(1/2)^{-1}\frac{|A|}{|B|}.$$

Let $B$ be of type $(T2)$. That is, $B \cong (\prod_{i \neq j_0} S_i) \oplus J(A)$ for some $j_0 \in \{1,...,r\}$. Again using Lemma \ref{deconstruct}, we have $$\frac{|A^{\times}|}{|B^{\times}|} = \frac{\prod_{i=1}^r|S_i^{\times}|}{\prod_{i \neq j_0}|S_i^{\times}|}=|S_{j_0}^{\times}|.$$ Again using Lemma \ref{babylemma0}, we have $$\phi(1/2)\frac{|A|}{|B|}=\phi(1/2)|S_{j_0}| < \frac{|A^{\times}|}{|B^{\times}|} < |S_{j_0}| =\frac{|A|}{|B|}.$$

Finally, let $B$ be of type $(T3)$. That is, $B \cong S \oplus H$ where $H$ is a two-sided ideal of $A$ that is maximal with respect to the condition $H \subset J(A)$. Then $$\frac{|A^{\times}|}{|B^{\times}|} =\frac{|J(A)|}{|H|}=\frac{|A|}{|B|}$$ by Lemma \ref{deconstruct} and since $J(B) \cong H$.

This completes the proof of the lemma.
\end{proof}
\end{lemma}

Let $x,y \in A$ be chosen uniformly at random. If $\langle x,y \rangle \neq A$ then $x$ and $y$ are both contained in a maximal subalgebra $B$ of $A$. For a given $B$, the probability that this occurs is $|B|^2/|A|^2$. Let $\Max A$ denote the set of maximal subalgebras of $A$. Then \begin{equation}\label{first} 1-P(A)=P(\langle x,y \rangle \neq A) \leq \sum_{B \in \Max A}|B|^2/|A|^2 .\end{equation}

Let $\mathcal{B}$ be the standard set of representatives of the conjugacy classes of maximal subalgebras of $A$. For a given $B \in \mathcal{B}$, there are $|A^{\times}|/|N_{A^{\times}}(B^{\times})|$ conjugates of $B$ in $A$. Combining $(\ref{first})$ with Lemma \ref{unitslemma} gives us \begin{equation}\label{second} 1-P(A) \leq \phi(1/2)^{-1}\sum_{B \in \mathcal{B}}(|A|/|B|)^{-1}=\phi(1/2)^{-1}\zeta_A(1). \end{equation} If $A$ is simple then, by Corollary \ref{zetatheoremcor}, $P(A) \to 1$ as $|A| \to \infty$. This completes the proof of Corollary \ref{randomgenerationcor}.

Let $c \in \R$ such that $1<c <q$ and let $\lambda >0$. For the general case, we need the assumption that $A$ is bounded by $(c,\lambda)$. Then, by Theorem \ref{zetatheorem} (and its proof), $P(A) \to 1$ as $n \to \infty$, as $m \to \infty$ or as $q \to \infty$. This completes the proof of Theorem \ref{randomgeneration}.

\section{Proof of Theorem \ref{minP}}\label{proofminP}

Let $A$ be a finite simple algebra. Write $A=M_n(q^m)$. Recall from $\S \ref{proofrandomgeneration}$ that \begin{equation}\label{second*} 1-P(A) \leq \sum_{B \in \Max A}\frac{|B|^2|A^{\times}|}{|A|^2|N_{A^{\times}}(B^{\times})|} \leq \phi(1/2)^{-1}\zeta_A(1). \end{equation} Since $\phi(1/2) \approx 0.2888$, it suffices to show that $\zeta_A(1) \leq 0.18$. We will first show that $\zeta_A(1) \leq 0.18$ if $n \neq 1$ and $(m,n) \neq (1,2)$, $(1,3)$, $(1,4)$ or $(2,2)$. We will then consider the remaining cases.

Recall from the proof of Theorem \ref{zetatheorem} that $$\zeta_A(1)\leq (2(n-1)+\omega(m))q^{-m(n-1)}.$$ It follows that $\zeta_A(1) \leq 0.18$ if $n\neq 1 $ and $(m,n) \neq (1,2)$, $(1,3)$, $(1,4)$, $(1,5)$, $(1,6)$, $(1,7)$, $(2,2)$, $(2,3)$, $(3,2)$ or $(4,2)$. For some of these remaining cases, we compute $\zeta_A(1)$ directly.

If $(m,n)=(1,5)$ then $\zeta_A(1)=2q^{-4}+2q^{-6}+q^{-20} \leq 0.16$.

If $(m,n)=(1,6)$ then $\zeta_A(1)=2q^{-5}+2q^{-8}+q^{-9}+q^{-18}+q^{-24} \leq 0.08$.

If $(m,n)=(1,7)$ then $\zeta_A(1)=2q^{-6}+2q^{-10}+2q^{-12}+q^{-42} \leq 0.04$.

If $(m,n)=(2,3)$ then $\zeta_A(1)=2q^{-4}+q^{-12}+q^{-9} \leq 0.13$.

If $(m,n)=(3,2)$ then $\zeta_A(1)=q^{-3}+q^{-6}+q^{-8} \leq 0.15$.

If $(m,n)=(4,2)$ then $\zeta_A(1)=q^{-4}+q^{-8}+q^{-8} \leq 0.08$.

\vspace{1.5mm} At this point, we have shown that $P(A) > 3/8$ if $n \neq 1$ and $(m,n)\neq (1,2)$, $(1,3)$, $(1,4)$ or $(2,2)$. For the remaining cases, we use other methods to bound $P(A)$. Define $$\nu_q(x):=\frac{1}{x}\sum\limits_{d|x}\mu(d)q^{x/d}$$ where $\mu$ is the M\"{o}bius function.

Let $n=1$. Let $G(A)$ be the set of generators of $A$ as a $k$-algebra. That is, $G(A)$ is the subset of $A$ consisting of all elements whose minimal polynomial over $k$ has degree $m$. It is a classical result, dating back to Gauss, that the number of monic irreducible polynomials over $k$ of degree $m$ is $\nu_q(m)$. So $|G(A)|=m\nu_q(m)$.
Hence $$P(A) \geq m\nu_q(m)q^{-m} \geq 1-q^{-1} \geq 1/2.$$

Now let $(m,n)=(1,2)$. Then, by Equation $(9)$ of \cite{KMP}, we have $$ P(A) = (q-1)(q^2-1)q^{-3} \geq 3/8$$ with equality if and only if $q=2$.

If $(m,n)=(2,2)$ then we have just shown that the probability of two randomly chosen elements of $A$ generating $A$ as a $\F_{q^2}$-algebra is strictly greater than $3/8$. So, certainly, $P(A)>3/8$ as a $\F_q$-algebra.

Let $(m,n)=(1,3)$. By Equation $(10)$ of \cite{KMP}, we have $$P(A) =(q^2-1)^2(q^3-1)q^{-7} \geq 63/128.$$

Let $(m,n)=(1,4)$. Then \begin{align*}\sum_{B \in \Max A}\frac{|B|^2|A^{\times}|}{|A|^2|N_{A^{\times}}(B^{\times})|}&
 = 2q^{-6}\frac{q^4-1}{q-1}+q^{-8}\frac{(q^4-1)(q^3-1)}{(q^2-1)(q-1)}+2^{-1}q^{-16}(q^4-q)(q^4-q^3) \\
& =2^{-1}q^{-12}\big(4q^9+6q^8+6q^7+8q^6+2q^5+3q^4-q^3-q+1 \big) \\ & \leq 0.61 .\end{align*} Hence $P(A) >3/8 $ by $(\ref{second*})$. This completes the proof.

\section{Proof of Theorem \ref{estimateprob}}\label{proofestimateprob}

Let $A$ be a finite simple algebra, say $A=M_n(q^m)$. Recall that $m(A)$ is the minimal index of any proper subalgebra of $A$. Note that $m(A)$ is undefined if $m=n=1$.

\begin{lemma}\label{minindexlemma1} Let $\mathcal{C}$ be the set of conjugacy classes of subalgebras of $A$ that have index $m(A)$. If $m>1$ then let $p$ be the smallest prime divisor of $m$. Then $m(A)$ and $\mathcal{C}$ are as follows:
\begin{table}[ht]
\centering
\begin{tabular}{c | c  c  c}
 & $m(A)$ & $|\mathcal{C}|$ & standard reps of $\mathcal{C}$ \\ [0.2ex] \hline
$n>2$ & $q^{m(n-1)}$ & $2$ & $P_{1,n-1}(q^m)$, $P_{n-1,1}(q^m)$ \\
 $n=2$ & $q^m$ & $1$ & $P_{1,1}(q^m)$ \\
$n=1$, $m>1$ & $q^{m(1-1/p)}$ & $1$ & $q^{m/p}$ \\
\end{tabular}
\end{table}\label{table1}
\begin{proof} Let $B$ be a subalgebra of $A$ with index $m(A)$. Then $B$ is maximal, so we refer to the classification in Theorem \ref{classthm1}.

We first assume that $n=1$. There do not exist any subalgebras of $A$ of type $(S1)$ or $(S2)$. So $B \cong q^{m/p}$, where $p$ is the smallest prime divisor of $m$. There is one conjugacy class of such a $B$.

Now assume that $n>1$. Observe that $\dim B$ divides $\dim A$ if $B$ is of type $(S2)$ or $(S3)$, whilst $2\dim B > \dim A$ if $B$ is of type $(S1)$. So $B$ is of type $(S1)$, that is, $B$ is conjugate to $P_{l,n-l}(t)$ for some $1 \leq l < n$. We compute $[A:B]=q^{ml(n-l)}$. Hence $m(A)=q^{m(n-1)}$, which is realised when $B$ is conjugate to $P_{1,n-1}(q^m)$ or to $P_{n-1,1}(q^m)$. Finally, we note that $P_{1,n-1}(q^m)$ is not conjugate to $P_{n-1,1}(q^m)$ unless $n=2$ (in which case they are equal).
\end{proof}
\end{lemma}

Henceforth assume that $A$ is not a field. That is, $n>1$.

\begin{lemma}\label{minindexlemma2} Let $B$ be a subalgebra of $A$.

$(i)$ If $[A:B]<m(A)^{4/3}$ then $[A:B]=m(A)$.

$(ii)$ If $m(A)^{4/3} \leq [A:B]<m(A)^{5/3}$ then either $n=4,5$ or $6$ and $B$ is conjugate to $P_{2,n-2}(q^m)$ or $P_{n-2,2}(q^m)$, or $B$ is non-maximal in $A$ and is not over $\F_{q^m}$.
\begin{proof} Note that $m(A)=q^{m(n-1)}$ by Lemma \ref{minindexlemma1} (since $n>1$). We first consider the case where $B$ is maximal. By Theorem \ref{classthm1}, $B$ is of type $(S1)$, $(S2)$ or $(S3)$. We consider each of these possibilities.

Let $B$ be of type $(S1)$. That is, $B$ is conjugate to $P_{l,n-l}(q^m)$ for some positive integer $l <n$. Observe that $[A:B]=q^{ml(n-l)}$. If $l=1$ or $n-1$ then $[A:B]=m(A)$. If $n=4,5$ or $6$ and $l=2$ or $n-2$ then $m(A)^{4/3} \leq [A:B]<m(A)^{5/3}$. Otherwise, $[A:B] \geq m(A)^{5/3}$.

Now let $B$ be of type $(S2)$ or $(S3)$. Then $[A:B]=q^{mn^2(1-1/a)}$ for some prime $a$. So $[A:B] \geq q^{mn^2/2} \geq m(A)^{5/3}$.

We have shown that there exist no maximal subalgebras (and hence no subalgebras) $B$ of $A$ that satisfy $m(A)<[A:B]<m(A)^{4/3}$. This proves $(i)$.

Now assume (for a contradiction) that $B$ is a $\F_{q^m}$-subalgebra of $A$ that is not maximal (as a $\F_q$-subalgebra) and satisfies $m(A)^{4/3} \leq [A:B]<m(A)^{5/3}$. Let $M$ be a maximal subalgebra of $A$ that contains $B$. By the previous argument, either $M \cong P_{1,n-1}(q^m)$ or $M \cong P_{2,n-2}(q^m)$ and $n=4,5$ or $6$.

Let $M \cong P_{2,n-2}(q^m)$ and $n=4,5$ or $6$. It follows from Theorems \ref{classthm1} and \ref{classthm2} that the minimal index of a subalgebra of $M$ is $q^m$. Then $[A:B] \geq q^{2m(n-2)+m} \geq m(A)^{5/3}$, which is a contradiction.

Let $M \cong P_{1,n-1}(q^m)$. If $n=2$ then, using Theorems \ref{classthm1} and \ref{classthm2}, the minimal index of a $\F_{q^m}$-subalgebra of $M$ is $q^m$. Then $[A:B] \geq q^{2m} \geq m(A)^{5/3}$. If $n>2$ then, again using Theorems \ref{classthm1} and \ref{classthm2}, the minimal index of a $\F_{q^m}$-subalgebra of $M$ is $q^{m(n-2)}$. Then $[A:B] \geq q^{2m(n-2)+m(n-1)} \geq m(A)^{5/3}$. We have a contradiction, proving $(ii)$.
\end{proof}
\end{lemma}

Let $\{B_i\hspace{0.5mm}|\hspace{0.5mm}i=1,...,\alpha\}$ denote the set of maximal subalgebras of $A$. Let $\beta$ be the number of maximal subalgebras of $A$ with index $m(A)$. We arrange the $B_i$'s such that $B_i$ has index $m(A)$ if and only if $i \leq \beta$.

Let $\kappa:A \to \R$ be defined by $\kappa(A):=\beta m(A)^{-1}$. Note that $\sum_{1 \leq i \leq \beta}[A:B_i]^{-2}=\kappa(A) m(A)^{-1}$.

Let $x,y \in A$ be chosen uniformly at random. If $\langle x,y \rangle \neq A$ then $x$ and $y$ are both contained in a maximal subalgebra of $A$. For a given $B_i$, the probability that this occurs is $|B_i|^2/|A|^2$. Then, as in $\S \ref{proofrandomgeneration}$, we have \begin{equation}\label{upbound} 1-P(A) \leq \sum_{1 \leq i \leq \alpha}[A:B_i]^{-2} = \kappa(A) m(A)^{-1} + \sum_{\beta+1 \leq i \leq \alpha}[A:B_i]^{-2} .\end{equation} Using the inclusion-exclusion principle, we obtain \begin{equation}\label{lowbound} 1-P(A) \geq \kappa(A) m(A)^{-1}-\sum_{1 \leq i<j \leq \beta}[A:B_i \cap B_j]^{-2}.\end{equation}

Let $\xi=\xi(n)$ be defined by $\xi=2$ if $n>2$ and $\xi=1$ if $n=2$.

\begin{lemma}\label{minindexlemma3} $\beta =\xi (q^{mn}-1)/(q^m-1)$.
\begin{proof} Recall from Lemma \ref{minindexlemma1} that $\{B_i\hspace{0.5mm}|\hspace{0.5mm}i=1,...,\beta\}$ splits into $\xi$ conjugacy classes. Let $i \in \{1,...,\beta\}$. Again by Lemma \ref{minindexlemma1}, recall that $B_i \cong P_{1,n-1}(q^m)$. So $B_i^{\times}$ is self-normalising in $A^{\times}$. Hence there are $|A^{\times}|/|B_i^{\times}|=(q^{mn}-1)/(q^m-1)$ conjugates of $B_i$ in $A$.
\end{proof}
\end{lemma}

We are now able to bound $\kappa(A)$.

\begin{corollary}\label{minindexlemma4} $1<\kappa(A) < 4$.
\begin{proof} Observe that $\kappa(A)=\xi q^{-m(n-1)}(q^{mn}-1)/(q^m-1)$ by Lemmas \ref{minindexlemma1} and \ref{minindexlemma3}. It is then easy to check that $1< \kappa(A) <4$.
\end{proof}
\end{corollary}

Note that the bounds in Corollary \ref{minindexlemma4} are best possible. For example, if $n=2$ then $\kappa(A) \to 1$ as $q \to \infty$ or as $m \to \infty$. If $q=2$ and $m=1$ then $\kappa(A) \to 4$ as $n \to \infty$.

It remains to estimate the final term in both of the inequalities $(\ref{upbound})$ and $(\ref{lowbound})$.

\begin{lemma}\label{minindexlemma5} $\sum_{\beta+1 \leq i \leq \alpha}[A:B_i]^{-2}=O(m(A)^{-4/3})$.
\begin{proof} Let $\mathcal{B}$ be the standard set of representatives of the conjugacy classes of maximal subalgebras of $A$. Let $\mathcal{B}_0$ be the subset of $\mathcal{B}$ consisting of subalgebras with index $m(A)$. Let $B \in \mathcal{B} \setminus \mathcal{B}_0$. Note that there are $[A^{\times}:N_{A^{\times}}(B^{\times})]$ conjugates of $B$ in $A$.

Let $\rho(A)$ denote the number of conjugacy classes of maximal subalgebras of $A$. Observe that $\rho(A)=n-1+\omega(n)+\omega(m)$ by Theorem \ref{classthm1}. Recall from Lemma \ref{minindexlemma1} that $m(A)=q^{m(n-1)}$. If $m(A) \to \infty$ then at least one of the following occurs: $n \to \infty$, $m \to \infty$ or $q \to \infty$. So $\rho(A)m(A)^{-1/3} \to 0$ as $m(A) \to \infty$. That is, \begin{equation}\label{uujj}\rho(A)=o(m(A)^{1/3}).\end{equation}

Combining $(\ref{uujj})$ with Lemmas \ref{unitslemma} and \ref{minindexlemma2} gives us \begin{align*}\sum_{\beta+1 \leq i \leq \alpha}[A:B_i]^{-2} &=\sum_{B \in \mathcal{B} \setminus \mathcal{B}_0} [A:B]^{-2}[A^{\times}:N_{A^{\times}}(B^{\times})] \\  &< \phi(1/2)^{-1}\big(2m(A)^{-4/3} + \rho(A)m(A)^{-5/3}\big) \\ &=O(m(A)^{-4/3}) \qedhere.\end{align*}
\end{proof}
\end{lemma}

We note that the constant $4/3$ in Lemma \ref{minindexlemma5} is best possible. For example, consider the case where $n=4$ and $B=P_{2,2}(q^m)$. Then $m(A)=q^{3m}$ and $[A:B]=q^{4m}$.

\begin{lemma}\label{minindexlemma6} $\sum_{1 \leq i<j \leq \beta}[A:B_i \cap B_j]^{-2}=O(m(A)^{-4/3})$.
\begin{proof} Fix $i,j$ such that $1 \leq i<j \leq \beta$. By Lemma \ref{minindexlemma1}, $B_i$ and $B_j$ are both over $\F_{q^m}$. So $B_i \cap B_j$ is a $\F_{q^m}$-algebra that is not maximal in $A$. Hence $[A:B_i \cap B_j] \geq m(A)^{5/3}$ by Lemma \ref{minindexlemma2}. Then $$ \sum_{1 \leq i<j \leq \beta}[A:B_i \cap B_j]^{-2} \leq \beta^2m(A)^{-10/3} < 16m(A)^{-4/3}$$ using Corollary \ref{minindexlemma4}.
\end{proof}
\end{lemma}

The theorem then follows from combining the inequalities $(\ref{upbound})$ and $(\ref{lowbound})$ with Corollary \ref{minindexlemma4} and Lemmas \ref{minindexlemma5} and \ref{minindexlemma6}.

We conclude this section with the following estimate of the zeta function of $A$. Let $\epsilon >0$. By the same argument as in the proof of Lemma \ref{minindexlemma5}, it is easy to see that $\rho(A)=o(m(A)^{\epsilon/3})$. Combining this with Lemmas \ref{minindexlemma1} and \ref{minindexlemma2} gives us \begin{equation}\label{zetaestimate}\zeta_A(\epsilon)=\delta(A) m(A)^{-\epsilon}+O(m(A)^{-4\epsilon/3})\end{equation} where $\delta:A \to \R$ is a function given by $\delta(A)=1$ if $n=2$ and $\delta(A)=2$ otherwise.

\section{Proof of Theorem \ref{nilpotentrandomgeneration} and Corollary \ref{nilpotentrandomgenerationcor}}\label{proofnilpotentrandomgeneration}

Let $A$ be a finite algebra, say $A=S \oplus J(A)$ where $S=\prod_{i=1}^rS_i$ is semisimple and $S_i=M_{n_i}(q^{m_i})$ for each $i$. Denote $n:=\min_{i=1,...,r}\{n_i\}$ and $m:=\min_{i=1,...,r}\{m_i\}$. If $A$ is simple, note that $n=1$ if and only if $A$ is a field. Let $T$ denote the group of scalar matrices of $S^{\times}$. Recall that $\phi(1/2) \approx 0.2888$.

Assume that $n>1$. We will need the following lemma.

\begin{lemma}\label{nilpotentlemma} Let $B$ be a maximal subalgebra of $A$. Then $\frac{|B^N|^2|A^{\times}|}{|A^N|^2|N_{A^{\times}}(B^{\times})|} < \phi(1/2)^{-1}\big(\frac{|A|}{|B|}\big)^{-\frac{1}{4}}$.
\begin{proof} We first consider the case where $A$ is simple. That is, $A=M_n(q^m)$. For simplicity, denote $t:=q^m$. By Theorem \ref{classthm1}, $B$ is of type $(S1)$, $(S2)$ or $(S3)$. We consider individually each of these possibilities. We will repeatedly use the fact that $|A^N|=t^{n^2-n}$, which was proved in Theorem $1$ of \cite{FH}.

Let $B$ be of type $(S1)$. That is, $B \cong P_{l,n-l}(t)$ for some positive integer $l <n$. Observe that $|B^N|=|(B/J(B))^N||J(B)|$ by Lemma \ref{deconstruct}. Then we have $$\frac{|B^N|}{|A^N|}= \frac{t^{l^2-l}\cdot t^{(n-l)^2-(n-l)}\cdot t^{l(n-l)}}{t^{n^2-n}}=t^{-l(n-l)}=\frac{|B|}{|A|}.$$

Let $B$ be of type $(S2)$. That is, $B \cong M_{n/a}(t^a)$ for some prime divisor $a$ of $n$. Then $$\frac{|B^N|}{|A^N|}=\frac{t^{a(n^2/a^2-n/a)}}{t^{n^2-n}}=t^{-n^2(1-1/a)}=\frac{|B|}{|A|}.$$

Hence, by Lemma \ref{unitslemma}, we have $$\frac{|B^N|^2|A^{\times}|}{|A^N|^2|N_{A^{\times}}(B^{\times})|} < \phi(1/2)^{-1}\Big(\frac{|A|}{|B|}\Big)^{-1}$$ for all $B$ of type $(S1)$ or $(S2)$.

Let $B$ be of type $(S3)$. That is, $B \cong M_n(t^{1/b})$ for some prime divisor $b$ of $m$. Observe that $N_{A^{\times}}(B^{\times})=B^{\times}T$, and so $|N_{A^{\times}}(B^{\times}):B^{\times}|=(t-1)/(t^{1/b}-1) \geq t^{1-1/b}$ (using $(\ref{elementary})$). Then \begin{align*}\frac{|B^N|^2|A^{\times}|}{|A^N|^2|N_{A^{\times}}(B^{\times})|} & < \frac{t^{2(n^2-n)/b}}{t^{2(n^2-n)}} \cdot \phi(1/2)^{-1}t^{n^2(1-1/b)} \cdot t^{1/b-1} \\ &= \phi(1/2)^{-1}t^{-(1-1/b)(n^2-2n+1)} \cdot  \\ & \leq \phi(1/2)^{-1}t^{-(1-1/b)n^2/4}\\ &=\phi(1/2)^{-1}\Big(\frac{|A|}{|B|}\Big)^{-\frac{1}{4}}.\end{align*} by Lemma \ref{unitslemma} and since $n>1$. This completes the proof for the case where $A$ is simple.

We now consider the general case. Recall that $A=S \oplus J(A)$ where $S=\prod_{i=1}^rS_i$ is semisimple. By Theorem \ref{classthm2}, $B$ is of type $(T1)$, $(T2)$ or $(T3)$. We consider each of these possibilities.

Let $B$ be of type $(T1)$. That is, $B \cong (B_j \times \prod_{i \neq j} S_i) \oplus J(A)$ for some $j \in \{1,...,r\}$ and maximal subalgebra $B_j$ of $S_j$. Then, using Lemma \ref{deconstruct}, we have $$\frac{|B^N|^2|A^{\times}|}{|A^N|^2|N_{A^{\times}}(B^{\times})|} \leq \frac{|B_j^N|^2|S_j^{\times}|}{|S_j^N|^2|N_{S_j^{\times}}(B_j^{\times})|}< \phi(1/2)^{-1}\Big(\frac{|S_j|}{|B_j|}\Big)^{-\frac{1}{4}} =\phi(1/2)^{-1}\Big(\frac{|A|}{|B|}\Big)^{-\frac{1}{4}}.$$

Let $B$ be of type $(T2)$. That is, $B \cong (\prod_{i \neq j_0} S_i) \oplus J(A)$ for some $j_0 \in \{1,...,r\}$. For simplicity, denote $n_0:=n_{j_0}$, $m_0:=m_{j_0}$ and $t_0:=q^{m_0}$. So $S_{j_0}=M_{n_0}(t_0)$. Observe that $N_{A^{\times}}(B^{\times})=B^{\times}T$, and so $|N_{A^{\times}}(B^{\times}):B^{\times}|=|Z(S_{j_0}^{\times})|=t_0-1$. Then $$\frac{|B^N|^2|A^{\times}|}{|A^N|^2|N_{A^{\times}}(B^{\times})|} = \frac{|S_{j_0}^{\times}|}{|S_{j_0}^N|^2(t_0-1)} =\frac{\prod_{i=0}^{n_0-1}(t_0^{n_0}-t_0^i)}{t_0^{2(n_0^2-n_0)}(t_0-1)} \leq 2t_0^{-n_0^2/4}= 2\Big(\frac{|A|}{|B|}\Big)^{-\frac{1}{4}}$$ using $(\ref{elementary})$ and since $n_0>1$.

Finally, let $B$ be of type $(T3)$. That is, $B \cong S \oplus H$ where $H$ is a two-sided ideal of $A$ that is maximal with respect to the condition $H \subset J(A)$. Then $$\frac{|B^N|^2|A^{\times}|}{|A^N|^2|N_{A^{\times}}(B^{\times})|} \leq \frac{|H|}{|J(A)|} =\Big(\frac{|A|}{|B|}\Big)^{-1}$$ by Lemma \ref{deconstruct}. This completes the proof of the lemma.
\end{proof}
\end{lemma}

Let $x,y \in A^N$ be chosen uniformly at random. If $\langle x,y \rangle \neq A$ then $x$ and $y$ are both contained in a maximal subalgebra $B$ of $A$. For a given $B$, the probability that this occurs is $|B^N|^2/|A^N|^2$. Let $\Max A$ denote the set of maximal subalgebras of $A$. Then \begin{equation}\label{firstnil} 1-P_N(A)=P(\langle x,y \rangle \neq A) \leq \sum_{B \in \Max A}|B^N|^2/|A^N|^2 .\end{equation}

Let $\mathcal{B}$ be the standard set of representatives of the conjugacy classes of maximal subalgebras of $A$. For a given $B \in \mathcal{B}$, recall that there are $|A^{\times}|/|N_{A^{\times}}(B^{\times})|$ conjugates of $B$ in $A$. Combining $(\ref{firstnil})$ with Lemma \ref{nilpotentlemma} gives us \begin{equation}\label{secondnil} 1-P_N(A) <\phi(1/2)^{-1}\sum_{B \in \mathcal{B}}(|A|/|B|)^{-1/4}=\phi(1/2)^{-1}\zeta_A(1/4). \end{equation} If $A$ is simple then, by Corollary \ref{zetatheoremcor}, $P_N(A) \to 1$ as $|A| \to \infty$. This completes the proof of Corollary \ref{nilpotentrandomgenerationcor}.

Let $c \in \R$ such that $1<c <q^{1/4}$ and let $\lambda >0$. For the general case, we need the assumption that $A$ is bounded by $(c,\lambda)$. Then, by Theorem \ref{zetatheorem} (and its proof), $P_N(A) \to 1$ as $n \to \infty$, as $m \to \infty$ or as $q \to \infty$. This completes the proof of Theorem \ref{nilpotentrandomgeneration}.

\section{Proof of Theorem \ref{polyrandomgeneration}}\label{proofpolyrandomgeneration}

Let $A$ be a finite simple algebra that is not a field. That is, $A=M_n(q^m)$ where $n>1$. For simplicity, denote $t:=q^m$.

Let $f$ be a polynomial of degree $n$ over $\F_t$. Factorise $f=f_1^{\alpha_1}f_2^{\alpha_2}...f_s^{\alpha_s}$ where the $f_i$'s are distinct and irreducible over $\F_t$. For each $i$, let $d_i$ be the degree of $f_i$. Without loss of generality, we assume that $f$ is monic.

For positive integers $u,v$, recall the definition $F(u,v)=(1-u^{-1})(1-u^{-2})...(1-u^{-v})$ and $F(u,0)=1$. We will need Theorem $2$ of \cite{Re}, which states that $$|A_f|=t^{n^2-n}\frac{F(t,n)}{\prod_{i=1}^s F(t^{d_i},\alpha_i)}=\frac{t^{-n}|A^{\times}|}{\prod_{i=1}^s F(t^{d_i},\alpha_i)}.$$ %

\begin{lemma}\label{polylemma} Let $B$ be a maximal subalgebra of $A$. There exists an absolute constant $C>0$ such that $\frac{|B_f|^2|A^{\times}|}{|A_f|^2|N_{A^{\times}}(B^{\times})|} \leq C\big(\frac{|A|}{|B|}\big)^{-\frac{1}{4}}$.
\begin{proof} By Theorem \ref{classthm1}, $B$ is of type $(S1)$, $(S2)$ or $(S3)$. We consider individually each of these possibilities. If $B_f$ is empty then we are done, so assume otherwise.

Let $B$ be of type $(S1)$. That is, $B \cong P_{l,n-l}(t)$ for some positive integer $l \leq n/2$. Let $\Lambda$ be the set of polynomials over $\F_t$ that divide $f$ and have degree $l$. We can assume that $\Lambda$ is non-empty (as otherwise $B_f$ is empty). Observe that $|\Lambda| \leq {n \choose l}$. Consider a generic element $f_0 \in \Lambda$. Factorise $f_0=f_1^{\beta_1}f_2^{\beta_2}...f_s^{\beta_s}$ where $0 \leq \beta_i\leq \alpha_i$ for each $i$. Then \begin{align*}|B_f| & \leq \sum_{f_0 \in \Lambda}|M_l(t)_{f_0}||M_{n-l}(t)_{f/f_0}||J(B)| \\ &=\sum_{f_0 \in \Lambda}\frac{t^{-l}|M_l(t)^{\times}|t^{-(n-l)}|M_{n-l}(t)^{\times}||J(B)|}{\prod_{i=1}^sF(t^{d_i},\beta_i)\prod_{i=1}^s F(t^{d_i},\alpha_i-\beta_i)} \\ & \leq |\Lambda|\Big(\frac{3}{2}\Big)^{n/2}\frac{t^{-n}|B^{\times}|}{\prod_{i=1}^s F(t^{d_i},\alpha_i)} \\ & \leq {n \choose l}\Big(\frac{3}{2}\Big)^{n/2}\frac{|B^{\times}||A_f|}{|A^{\times}|} \end{align*} using Lemmas \ref{deconstruct}, \ref{babylemma1} and Theorem $2$ of \cite{Re}. 

For sufficiently large $n$, say $n \geq 200$, observe that \begin{equation}\label{iijjii} 2l \log_2 n +n\log_2(3/2) \leq 3l(n-l)/4. \end{equation} Let $C=3\cdot 199^{199}\big(\frac{3}{2}\big)^{199}$. Then, using $(\ref{iijjii})$ and Lemma \ref{unitslemma}, we have \begin{align*}\frac{|B_f|^2|A^{\times}|}{|A_f|^2|N_{A^{\times}}(B^{\times})|} & \leq {n \choose l}^2 \Big(\frac{3}{2}\Big)^{n}\frac{|B^{\times}|}{|A^{\times}|} \\ & < 3n^{2l}\Big(\frac{3}{2}\Big)^{n}\frac{|B|}{|A|} \\& = 3t^{2l \log_t n +n \log_t(3/2)-l(n-l)}  \\&\leq Ct^{-l(n-l)/4} \\ & = C\big(\frac{|A|}{|B|}\big)^{-\frac{1}{4}}.\end{align*}

Let $B$ be of type $(S2)$. That is, $B \cong M_{n/a}(t^a)$ for some prime divisor $a$ of $n$. Let $z \in B_f$. Recall that $f$ is the characteristic polynomial of $z$ as a $n \times n$ matrix over $\F_t$. Let $g$ be the characteristic polynomial of $z$ as a $n/a \times n/a$ matrix over $\F_{t^a}$. Let $\Gamma_a := \Gal(\F_{t^a}/\F_t) \cong \Z_a$.

Without loss of generality, we rearrange the factors of $f$ such that, for some positive integer $c \leq s$, $f_i$ is reducible over $\F_{t^a}$ if and only if $i\leq c$.

Let $i \in \{1,...,s\}$. Let $g_i$ be a $\F_{t^a}$-irreducible factor of $f_i$. If $i>c$ then $f_i=g_i$. If $i \leq c$ then, since $a$ is prime, $f_i=\prod_{\sigma \in \Gamma_a}g_i^{\sigma}$ where the $\Gamma_a$-conjugates of $g_i$ are all distinct. So the polynomials in the set $\{g_i^{\sigma} \hspace{0.5mm} | \hspace{0.5mm} i =1,...,c; \sigma \in \Gamma_a\} \cup \{g_i \hspace{0.5mm} | \hspace{0.5mm}i=c+1,...,s\}$ are all $\F_{t^a}$-irreducible and distinct.

Let $p_i$ be the greatest common divisor of $f_i^{\alpha_i}$ and $g$. Note that $f=\prod_{\sigma \in \Gamma_a}g^{\sigma}$ by Lemma $5.1$ of \cite{NP}. 
So if $i>c$ then $p_i=g_i^{\alpha_i/a}$ and if $i \leq c$ then $p_i=\prod_{\sigma \in \Gamma_a}(g_i^{\sigma})^{_{\sigma\hspace{-0.7mm}}\gamma_i}$ where each $_{\sigma\hspace{-0.7mm}}\gamma_i$ is a non-negative integer such that $\sum_{\sigma \in \Gamma_a} {}_{\sigma\hspace{-0.7mm}}\gamma_i=\alpha_i$. Given that $f$ is fixed, observe that there are at most $a^{\frac{n}{a}}$ possibilities for $g$ (by allowing $\sum_{\sigma \in \Gamma_a} {}_{\sigma\hspace{-0.7mm}}\gamma_i=\alpha_i$ to range over all partitions for each $i \leq c$).

Applying Theorem $2$ of \cite{Re}, we have \begin{align*}|B_f| & \leq a^{\frac{n}{a}}\frac{t^{-n}|B^{\times}|}{\prod_{i=1}^c \prod_{\sigma \in \Gamma_a}F(t^{d_i},{}_{\sigma\hspace{-0.7mm}}\gamma_i)\prod_{i=c+1}^s F(t^{d_ia},\alpha_i/a)} \\ & \leq n^{\frac{n}{2}}2^{\alpha_1+...+\alpha_c}\frac{t^{-n}|B^{\times}|}{\prod_{i=1}^s F(t^{d_i},\alpha_i)} \\ & \leq (2n)^{\frac{n}{2}}\frac{|B^{\times}||A_f|}{|A^{\times}|}.\end{align*}

Observe that $n\log_2(2n) \leq n^2/4$ for $n \geq 22$. Then \begin{align*}\frac{|B_f|^2|A^{\times}|}{|A_f|^2|N_{A^{\times}}(B^{\times})|} & \leq (2n)^n \frac{|B^{\times}|}{|A^{\times}|} \\ & <3 (2n)^n \frac{|B|}{|A|} \\ & = 3t^{n\log_t(2n)-n^2(1-1/a)} \\ & \leq Ct^{-n^2(1-1/a)/2}\\ & = C\big(\frac{|A|}{|B|}\big)^{-\frac{1}{2}}\end{align*} using Lemma \ref{unitslemma} and since $C \geq 3(2\cdot 21)^{21}$.

Let $B$ be of type $(S3)$. That is, $B \cong M_n(t^{1/b})$ for some prime divisor $b$ of $m$. Let $\Gamma_b := \Gal(\F_{t}/\F_{t^{1/b}}) \cong \Z_b$. We assume that $f$ is over $\F_{t^{1/b}}$ (as otherwise $B_f$ is empty). That is, $f$ is $\Gamma_b$-stable. 

Since $b$ is prime, each factor $f_i$ of $f$ is either over $\F_{t^{1/b}}$ or $\prod_{\sigma \in \Gamma_b}f_i^{\sigma}$ is $\F_{t^{1/b}}$-irreducible where the $\Gamma_b$-conjugates of $f_i$ are all $\F_{t}$-irreducible and distinct. 
Let $d:=|\{1 \leq i \leq s \hspace{0.5mm}|\hspace{0.5mm} f_i \textnormal{ is over }\F_{t^{1/b}}\}|$. Since $f$ is $\Gamma_b$-stable, we can rearrange the factors of $f$ such that $f_i$ is over $\F_{t^{1/b}}$ if and only if $i \leq d$, $b$ divides $s-d$ and, for every $i=1,...,(s-d)/b$, $\prod_{j=0}^{b-1}f_{d+i+j(s-d)/b}=\prod_{\sigma \in \Gamma_b}f_{d+i}^{\sigma}$ and $\alpha_{d+i}=\alpha_{d+i+(s-d)/b}=...=\alpha_{d+i+(b-1)(s-d)/b}$.

For $i \in \{1,...,d+(s-d)/b\}$, define a polynomial $h_i$ by $h_i=f_i$ if $i \leq d$ and $h_i=\prod_{\sigma \in \Gamma_b}f_i$ otherwise. Observe that the $h_i$'s are all distinct and $\F_{t^{1/b}}$-irreducible. Then $$|B_f| =\frac{t^{-n/b}|B^{\times}|}{\prod_{i=1}^d F(t^{d_i/b},\alpha_i)\prod_{i=d+1}^{d+(s-d)/b}F(t^{d_i},\alpha_i)} \leq \frac{t^{-n/b}|B^{\times}|}{\prod_{i=1}^s F(t^{d_i/b},\alpha_i)}$$ by Theorem $2$ of \cite{Re} and Lemma \ref{babylemma}. Recall that $|N_{A^{\times}}(B^{\times}):B^{\times}|=(t-1)/(t^{1/b}-1)$. Then \begin{align*} \frac{|B_f|^2|A^{\times}|}{|A_f|^2|N_{A^{\times}}(B^{\times})|}  &= \frac{t^{-2n/b}(t^{1/b}-1)\prod_{i=1}^s F(t^{d_i},\alpha_i)^2|B^{\times}|}{t^{-2n}(t-1)\prod_{i=1}^s F(t^{d_i/b},\alpha_i)^2|A^{\times}|} \\& < 6t^{(2n-1)(1-1/b)}\frac{\prod_{i=1}^s F(t^{d_i},\alpha_i)^2|B|}{\prod_{i=1}^s F(t^{d_i/b},\alpha_i)^2|A|} \\ & \leq 3\cdot 2^{2n+1}t^{(2n-1-n^2)(1-1/b)} \\ & \leq \begin{cases} 96t^{-n^2(1-1/b)/4} & \text{if } n=2 \\  384t^{(6n-13-n^2)(1-1/b)} & \text{if } n>2\\ \end{cases} \\ & \leq 384\big(\frac{|A|}{|B|}\big)^{-\frac{1}{4}}\end{align*} using $(\ref{elementary})$ and Lemmas \ref{babylemma} and \ref{unitslemma}.

This proves the lemma, taking $C=3\cdot 199^{199}\big(\frac{3}{2}\big)^{199}$.
\end{proof}
\end{lemma}

Let $x,y \in A_f$ be chosen uniformly at random. If $\langle x,y \rangle \neq A$ then $x$ and $y$ are both contained in a maximal subalgebra $B$ of $A$. For a given $B$, the probability that this occurs is $|B_f|^2/|A_f|^2$. Let $\Max A$ denote the set of maximal subalgebras of $A$. Then \begin{equation}\label{firstpoly} 1-P_f(A)=P(\langle x,y \rangle \neq A) \leq \sum_{B \in \Max A}|B_f|^2/|A_f|^2 .\end{equation}

Let $\mathcal{B}$ be the standard set of representatives of the conjugacy classes of maximal subalgebras of $A$. For a given $B \in
\mathcal{B}$, recall that there are $|A^{\times}|/|N_{A^{\times}}(B^{\times})|$ conjugates of $B$ in $A$. Combining $(\ref{firstpoly})$ with Lemma \ref{polylemma} gives us \begin{equation}\label{secondpoly} 1-P_f(A) \leq C\sum_{B \in \mathcal{B}}(|A|/|B|)^{-1/4}=C\zeta_A(1/4)\end{equation} for some absolute constant $C>0$. Hence, by Corollary \ref{zetatheoremcor}, $P_f(A) \to 1$ as $|A| \to \infty$.

\section{Proof of Theorem \ref{rankrandomgeneration} and Corollary \ref{rankrandomgenerationcor}}\label{proofrankrandomgeneration}

Let $A$ be a finite simple algebra, say $A=M_n(q^m)$, where $n \ge 2$ and $m \ge 1$. Let $p$ be the smallest prime divisor of $n$. Let $\alpha:=\alpha(n)$ be a positive integer such that $\alpha \leq n$.

For simplicity, denote $t:=q^m$. It is a classical result, dating back to \cite{La}, that \begin{equation}\label{rankii} |A_{\alpha}|=\prod_{i=0}^{\alpha-1} \frac{(t^n-t^i)^2}{t^{\alpha}-t^i}.\end{equation}

We first prove part $(i)$ of the theorem. In part $(i)$ we consider $n$, and hence $\alpha$, to be fixed constants. Assume that $n \geq p\alpha$.

Let $B$ be a subalgebra of $A$ such that $B \cong M_{n/p}(t^p)$. Such a $B$ exists and is maximal by Theorem \ref{classthm1}. We claim that \begin{equation}\label{woot}\frac{|B_{\alpha}|}{|A_{\alpha}|} \ge t^{-p\alpha^2}.\end{equation}

We first consider the case where $\alpha=1$. Then, using $(\ref{rankii})$, we have $$ \frac{|B_{\alpha}|}{|A_{\alpha}|} = \frac{t-1}{t^{p}-1} \geq t^{-p}.$$ So indeed $(\ref{woot})$ holds. Next assume that $\alpha \neq 1$. Once again using $(\ref{rankii})$ gives us \begin{align*}\frac{|B_{\alpha}|}{|A_{\alpha}|} & =\prod_{i=0}^{\alpha-1} \frac{(t^n-t^{pi})^2(t^{\alpha}-t^i)}{(t^n-t^i)^2(t^{p\alpha}-t^{pi})} \\
& \ge \bigg(\frac{(t^{n}-t^{p(\alpha-1)})^2(t^{\alpha}-t^{\alpha-1})}{t^{2n}(t^{p\alpha}-t^{p(\alpha-1)})}\bigg)^{\alpha} \\ & = \big(t^{-(p-1)\alpha}(1-t^{-p})(1-t^{-1})\big)^{\alpha} \\ & \geq t^{-p\alpha^2}.\end{align*} So we have established $(\ref{woot})$. Hence $$P_{\alpha}(A) \leq 1-\frac{|B_{\alpha}|^2}{|A_{\alpha}|^2} \leq 1- t^{-2p\alpha^2}.$$

We now move on to part $(ii)$ of the theorem. We no longer consider $\alpha$ to be a constant, but rather an integer-valued function of $n$, which can vary. Assume that $n - \sqrt{n}/3 \le \alpha \leq n$.

Let $(K)$ be a property of elements of $A$. Let $E$ (resp. $E^K$) be the event that two elements of $A$ chosen uniformly at random generate $A$ (resp. both have property $(K)$). Let $P(E|E^K)$ denote the probability that two random elements of $A$ with property $(K)$ generate $A$.

\begin{lemma}\label{ranklemma} $P(E|E^K) \geq 1-\frac{2(2n-2+\omega(m))q^{-mn/4}}{P(E^K)}$.
\begin{proof} Using elementary probability theory, we have $$P(E|E^K)=\frac{P(E \cap E^K)}{P(E^K)}\geq 1-\frac{1-P(E)}{P(E^K)}.$$ Then $$P(E|E^K) \geq 1-\frac{2\zeta_A(1/2)}{P(E^K)}\geq 1- \frac{2(2n-2+\omega(m))q^{-mn/4}}{P(E^K)}$$ using $(\ref{second})$ and the proof of Theorem \ref{zetatheorem}.
\end{proof}
\end{lemma}

Let $x \in A$ be chosen uniformly at random. Note that the probability that $x$ is invertible is at least $1/4$. Recall that $\alpha \leq n$. Then, using $(\ref{elementary})$ and $(\ref{rankii})$, we have $$P(\rk(x)=\alpha)=t^{-n^2}|A^{\alpha}| \geq t^{-n^2}\frac{|\hspace{-0.5mm}\GL_n(t)|}{\prod_{j=\alpha}^{n-1}(t^n-t^j)} \prod_{i=0}^{\alpha-1} \frac{(t^n-t^i)}{t^{\alpha}-t^i} \geq \frac{1}{4t^{(n-\alpha)^2}}.$$ We now apply Lemma \ref{ranklemma} where we take $(K)$ to be the property that an element of $A$ has rank $\alpha$. This gives us $$P_{\alpha}(A) \geq 1 - 32(2n-2+\omega(m))q^{-m(n/4-2(n-\alpha)^2)}.$$ Rearranging $n-\sqrt{n}/3 \le \alpha$ gives us $n/4-2(n-\alpha)^2 \ge n/36$, and hence $P_{\alpha}(A) \to 1$ as $|A| \to \infty$. This completes the proof of Theorem \ref{rankrandomgeneration}.

Recall that a matrix is invertible if and only if it has full rank. Then Corollary \ref{rankrandomgenerationcor} follows from applying Theorem \ref{rankrandomgeneration}$(ii)$ to the case where $\alpha=n$.

\section{The minimal number of generators}\label{EstimatingdA}

Let $d(A)$ be the minimal number of generators of a finite algebra $A$. Recall our convention that subalgebras of $A$ contain the multiplicative identity of $A$. For an ideal $I$ of $A$, we define $d(I)$ to be the minimal number of generators of $I$ as a non-unital algebra.

We begin with the following elementary observation. 

\begin{lemma}\label{easypeasy} Let $A$ be a finite algebra and let $I$ be an ideal of $A$. Then $$ d(A/I) \leq d(A) \leq d(A/I) +d(I) .$$
\begin{proof} Take the image/preimage of a generating set under the natural projection $A \to A/I$.
\end{proof}
\end{lemma}

We now characterise when $d(A)\leq 1$. Recall that $$\nu_q(x):=\frac{1}{x}\sum\limits_{d|x}\mu(d)q^{x/d}$$ where $\mu$ is the M\"{o}bius function.

\begin{lemma}\label{smalld} Let $A$ be a finite algebra. Then the following hold.

$(i)$ $d(A) =0$ if and only if $A=k$.

$(ii)$ If $d(A) =1$ then $\dim A >1$ and $A/J(A)=\prod_{i=1}^r (\F_{q^{m_i}})^{\alpha_i}$ where $1 \leq m_1 < ... <m_r$ and $\alpha_i \leq \nu_q(m_i)$ for each $i$. If $A$ is semisimple then the converse holds.

\begin{proof} $(i)$ We have $d(A)=0$ if and only if $A$ does not have a maximal subalgebra if and only if $A=k$.

$(ii)$ We first consider the case where $A$ is simple, say $A=M_n(q^m)$.

Let $d(A)=1$. Assume (for a contradiction) that $n>1$. Let $x$ be a generator of $A$. Let $\chi_n(x)$ be the characteristic polynomial of $X$ as a $n \times n$ matrix over $\F_{q^m}$. If $\chi_n(x)$ is $\F_{q^m}$-irreducible then, by Theorem $2.1$ of \cite{NP}, $\dim_k \F_{q^m}\langle x \rangle=mn$ 
and so $k\langle x \rangle$ is a proper subalgebra of $A$. If $\chi_n(x)$ is $\F_{q^m}$-reducible then $x$ is contained in a parabolic subalgebra of $A$. This is a contradiction, hence $n=1$. Conversely, let $A=\F_{q^m}$ for $m>1$. Any generator of the multiplicative group $A^{\times}$ then generates $A$ as an algebra.

Now consider the case where $A=S^{\alpha}$ for simple $S$.

Let $d(A)=1$. Then $S$ is a field by Lemma \ref{easypeasy} and the above arguments. Write $S=\F_{q^m}$. Recall from the proof of Theorem \ref{minP} that the number of generators of $S$ as a $k$-algebra is $m\nu_q(m)$. By Theorem $6.3$ of \cite{KMP}, $A$ can be generated by $1$ element if and only if $\alpha \leq \nu_q(m)$. The converse follows immediately.

Next consider the case where $A$ is semisimple, say $A = \prod_{i=1}^r S_i^{\alpha_i}$ where the $S_i$'s are pairwise non-isomorphic simple algebras. It follows from Proposition $2.12$ of \cite{KMP} that $d(A) = \max_{i=1,...,r} \{d(S_i^{\alpha_i})\}$. The result then follows from Lemma \ref{easypeasy} and the above arguments.

Finally, we consider the general case. If $d(A)=1$ then $d(A/J(A))\leq 1$ by Lemma \ref{easypeasy}. This completes the proof.
\end{proof}
\end{lemma}

\begin{corollary}\label{smalldcor} Let $A$ be a finite simple algebra. Then $$d(A)= \begin{cases} 2 &\mbox{if } A \textnormal{ is not a field} \\
1 & \mbox{if } A \textnormal{ is a field and }A\neq k  \\
0 & \mbox{if } A=k \end{cases}.$$
\begin{proof} By Theorem $6.4$ of \cite{KMP}, $A$ is $2$-generated. The result then follows immediately from Lemma \ref{smalld}.
\end{proof}
\end{corollary}

Let $P(A,l)$ be the probability that $l$ elements of $A$ chosen uniformly at random will generate $A$ as a $k$-algebra. Note that $P(A,l) \geq P(A,l_0)$ for all $l \geq l_0$. Recall our previous notation $P(A):=P(A,2)$.

\medskip

\noindent{\it Proof of Theorem \ref{mind}.}

\begin{proof} We first consider the case where $r=1$ and $J(A)=0$. Write $A=S^{\alpha}$ where $S=M_n(q^m)$.

If $d(A)=0$ then $A=k$ by Lemma \ref{smalld} and the result is immediate. If $d(A) = 1$ then $S$ is a field and $\alpha \leq \nu_q(m)$ by Lemmas \ref{easypeasy} and \ref{smalld}. Observe that $P(S,1)=m\nu_q(m)q^{-m} \leq 1$ as in the proof of Theorem \ref{minP}. So $f(A)=m^{-1}\log_q \alpha m  \leq m^{-1}\log_q P(S,1)+1 \leq 1$.

Henceforth assume that $d(A)\geq 2$. By Theorem $6.3$ of \cite{KMP}, $A$ can be generated by $l$ elements if and only if \begin{equation}\label{boundi} \alpha \leq \frac{q^{lmn^2}P(S,l)}{m|\!\PGL_n(q^m)|}.\end{equation}

Taking $l \geq 2$, we have $P(A,l) \geq P(A,2) \geq 3/8$ by Theorem \ref{minP}. Combining this with $(\ref{boundi})$ gives us $$d(A) \leq \left \lceil{m^{-1}n^{-2}\log_q \frac{8\alpha m}{3(q^m-1)}}\right \rceil +1 < m^{-1}n^{-2}\log_q \alpha m +3.42.$$

For the lower bound, combining $(\ref{boundi})$ with Lemma \ref{babylemma0} gives us $$d(A) > m^{-1}n^{-2}\log_q \frac{\phi(1/2)\alpha m }{q^m-1} +1 > m^{-1}n^{-2}\log_q \alpha m-2.33.$$

Now consider the case where $A$ is semisimple. That is, $A=\prod_{i=1}^r S_i^{\alpha_i}$ where the $S_i$'s are simple and pairwise non-isomorphic. It follows from Proposition $2.12$ of \cite{KMP} that $d(A) = \max_{i=1,...,r} \{d(S_i^{\alpha_i})\}$. The result follows immediately.

Finally, we consider the general case. That is, $A=S \oplus J(A)$ where $S$ is semisimple. The lower bound is immediate from Lemma \ref{easypeasy} and the semisimple case. For the upper bound, let $$0 =H_0 < H_1 < ... <H_{\mu}= J(A)$$ be an unrefinable chain of minimal length of $S$-subbimodules of $J(A)$. For each $i=1,...,\mu$, let $x_i \in H_i \setminus H_{i-1}$. Let $X$ be a generating set for $S$ of minimal cardinality. Using Theorem \ref{classthm2}, we see that $X \cup \{x_1,...,x_{\mu}\}$ is a set of generators for $A$. That is, $d(A) \leq d(S) + \mu$. This completes the proof of the theorem.
\end{proof}



\section{Positively finitely generated algebras}\label{PFGA}

In this section we investigate positively finitely generated profinite algebras.

Let $A$ be a profinite algebra. Recall the following definitions. For $d \ge 1$, $P(A,d)$ is the probability that $d$ randomly chosen elements of $A$ generate $A$. Let $m_n(A)$ be the number of index $n$ (open) maximal subalgebras of $A$.

In order to prove Theorems \ref{PFG} and \ref{PSAA} we need some preparations.

\begin{lemma} With the above notation we have $1-P(A,d) \le \sum_{n \ge 2} m_n(A)n^{-d}$.
\end{lemma}

\begin{proof} If randomly chosen $a_1, \ldots , a_d \in A$ do not generate $A$ (topologically)
then they all lie in some (open) maximal subalgebra $B$ of $A$. Therefore
$1 - P(A,d) \le \sum_{B \in \Max A} [A:B]^{-d}$ yielding the result.
\end{proof}

Given the algebra $A$ and a subalgebra $B < A$, define the \emph{core} $B_A$ of $B$ in $A$ to be
the maximal two-sided ideal $C$ of $A$ such that $C \subseteq B$. It exists (as the sum of all
ideals of $A$ which are contained in $B$) and it is unique.



\begin{lemma}\label{trivialcore} Let $A$ be a profinite algebra. Then, for all $n \ge 2$, $A$ has at most $6.93n$ maximal subalgebras of index $n$ with trivial core.
\end{lemma}

\begin{proof} By our assumptions, $A=S \oplus J(A)$ where $S=\prod_{i \in I}S_i$ is a semisimple subalgebra of $A$ such that each $S_i$ is simple.
Let $B$ be a maximal subalgebra of $A$ of finite index $n \ge 2$, and let $C = B_A$.

It is straightforward to generalise Theorem \ref{classthm2} to profinite algebras, so $B$ is of type $(T1)$, $(T2)$ or $(T3)$. If $B$ is of type $(T1)$ then $C = \smash{\big(\prod_{i \ne j} S_i\big) \oplus J(A)}$ for some $j \in I$. If $B$ is of type $(T2)$ then $C = \big(\prod_{i \ne j_1,j_2} S_i\big) \oplus J(A)$ where $j_1, j_2 \in I$ are distinct and $S_{j_1} \cong S_{j_2}$. If $B$ is of type $(T3)$ then $C = \big(\prod_{i \neq j_1,j_2} S_i\big) \oplus J(B)$ for some (not necessarily distinct) $j_1, j_2 \in I$.

Henceforth assume that $B$ has trivial core, namely $C=0$.

Suppose $B$ is of type $(T1)$. Then $|I| = 1$ and $J(A)$ is trivial. That is, $A$ is finite and simple. It follows from Theorem \ref{classthm1} that $A$ has at most two conjugacy classes of maximal subalgebras of index $n$.
Lemma \ref{unitslemma} shows that, given $B$ as above, $A$ has at most $\frac{|A^{\times}|}{|B^{\times}|} < \phi(1/2)^{-1}\frac{|A|}{|B|} =\phi(1/2)^{-1}n$ subalgebras which are conjugate to $B$. Note that $2\phi(1/2)^{-1} \approx 6.925$.

Now let $B$ be of type $(T2)$. Then $|I|=2$ and $J(A)$ is trivial. That is, $A = S_1 \times S_2$ where $S_1, S_2$ are isomorphic finite simple algebras, and $B$ is a diagonal subalgebra of $A$, so $n = |S_1| = |S_2|$. The number of choices for $B$ is therefore bounded above by the number of isomorphisms from $S_1$ to $S_2$, which in turn is bounded above by $n$.

Finally, suppose $B$ is of type $(T3)$. Then $|I| =1$ or $2$ and $J(A)$ is a simple $S$-bimodule with $|J(A)|=n$. By the Wedderburn-Malcev Principal Theorem, $B=S^{1+z}$ for some $z \in J(A)$. So there are precisely $n$ choices for $B$.

Altogether we see that the number of maximal subalgebras of $A$ of index $n$ with trivial core is at most $6.93n$.
\end{proof}


To illustrate Lemma \ref{trivialcore}, consider the case where $A = M_2(q)$. It is easy to check that $m_n(A) \le n+1$ for all $n > 1$.

Analogous to the Haar measure for locally compact groups, every profinite algebra admits a unique left (additive) translation invariant probability measure.

\begin{lemma}\label{probindep} Let $B_1, B_2$ be maximal subalgebras of $A$ with cores $C_1, C_2$ respectively.
Suppose $C_1 \ne C_2$ and let $d$ be a positive integer. Then the events $B_1^d, B_2^d$ in $A^d$
are independent.
\end{lemma}

\begin{proof} Replacing $A$ with $A/(C_1 \cap C_2)$ we may assume that $\dim A < \infty$.
Clearly $B_1^d, B_2^d$ are independent if and only if $B_1, B_2$ are, namely if and only if
$[A: B_1 \cap B_2] = [A:B_1][A:B_2]$ if and only if $\dim A = \dim B_1 + \dim B_2 - \dim (B_1 \cap B_2)$
if and only if $\dim A = \dim (B_1+B_2)$ if and only if $B_1+B_2 = A$.

Suppose $C_1 \ne C_2$. Without loss of generality, $C_2 \not\subset C_1$. Then $C_2$ is an ideal
of $A$ which is not contained in the maximal subalgebra $B_1$. Hence $B_1 + C_2$ is a subalgebra
of $A$ which properly contains $B_1$. It follows that $B_1 + C_2 = A$, which implies $B_1+B_2=A$.
We conclude that $B_1, B_2$ are independent.
\end{proof}

\medskip

\noindent
{\it Proof of Theorem \ref{PFG}.}

\begin{proof}

PMSG easily implies PFG.
Indeed, if $m_n(A) \le n^b$ for some positive integer $b$ and all $n \ge 2$, then
\[
1 - P(A, b+2) \le \sum_{n \ge 2} m_n(A)n^{-(b+2)} \le \sum_{n \ge 2} n^{-2} = \pi^2/6 - 1 < 1,
\]
so $P(A, b+2) > 0$.

Now, suppose $A$ is PFG, and let $d \in \N$ such that $P(A,d) > 0$.
We shall show that $A$ has PMSG.

Let $C_i$ be a list of the distinct cores of maximal subalgebras of $A$.
For each $i$ choose a maximal subalgebra $B_i$ of $A$ with core $C_i$.
For each $n \ge 2$ let $c_n(A)$ denote the number of maximal subalgebras
of index $n$ obtained in this way.

Consider $X = A^d$ as a probability space and the events $X_i = B_i^d < X$. By Lemma \ref{probindep} the events $X_i$ are pairwise independent.
Let $p_i = [A:B_i]^{-d}$, the probability of the event $X_i$.

By the Borel-Cantelli Lemma, if $\sum_i p_i = \infty$, then, with probability 1, infinitely many events $X_i$ occur.
This implies that a random $d$-tuple in $A^d$ generates $A$ with probability 0, a contradiction to $P(A,d) > 0$.
We conclude that $\sum_i p_i$ converges. Moreover, by the effective version of the Borel-Cantelli Lemma
we have
\[
\sum_i p_i \le P(A,d)^{-1}.
\]
We deduce that
\[
\sum_{n \ge 2} c_n(A)n^{-d} = \sum_i [A:B_i]^{-d} \le P(A,d)^{-1},
\]
so $c_n(A)n^{-d} \le P(A,d)^{-1}$, which yields
\[
c_n(A) \le P(A,d)^{-1} n^d
\]
for all $n \ge 2$.

Now, by Lemma \ref{trivialcore}, there are at most $6.93n$ maximal subgroups of $A$ of index $n$
with a given core $C_i$. This yields
\begin{equation}\label{strong}
m_n(A) \le 6.93nc_n(A) \le 6.93P(A,d)^{-1} n^{d+1}.
\end{equation}
In particular, $A$ has PMSG as required.

Finally, assume that $A$ is infinite and recall that $d_0(A):=\min\{d \ge 1 \hspace{0.5mm}|\hspace{0.5mm} P(A,d) > 0\}$. It then follows from equation $(\ref{strong})$ that
\[
M^*(A)= \limsup_{n > 1} \log m_n(A) / \log n \le d_0(A)+1,
\]
establishing the second statement of Theorem \ref{PFG}.
\end{proof}

We now move on to the proof of Theorem \ref{PSAA}.

\begin{propn}\label{mn} Let $A$ be a finite algebra, say $A=S \oplus J(A)$ where $S=\prod_{i=1}^rS_i$ and $S_i=M_{n_i}(q^{m_i})$ for each $i$. Then, for all $n > 1$ we have
\[
m_n(A) \le n\big(6.93r+r(r-1)/2+r^2n^{d(A)} \big).
\]
\end{propn}

\begin{proof} Let $B < A$ be a maximal subalgebra of index $n$ and let $C = B_A$ be its core.

If $B$ is of type $(T1)$ then $C = \big(\prod_{i \ne j} S_i\big) \oplus J(A)$ for some $1 \le j \le r$, so there are $r$ possibilities for $C$.
Given $C$, $B/C < A/C \cong S_j$ is a maximal subalgebra of index $n$, so by Lemma \ref{trivialcore} there are at most $6.93n$ possibilities for $B/C$, hence for $B$ given $C$. We conclude that there are at most $6.93rn$ possibilites for $B$ of type $(T1)$.

Suppose $B$ is of type $(T2)$. Then $C = \big(\prod_{i \ne j_1,j_2} S_i\big) \oplus J(A)$ for some $1 \le j_1 < j_2 \le r$, so there are $r(r-1)/2$
possibilities for $C$. As follows from the proof of Lemma \ref{trivialcore}, there are at most $n$ possibilities for $B$ given $C$.
Hence there are at most $nr(r-1)/2$ possibilities for $B$ in this case.

Finally, let $B$ be of type $(T3)$. Then $C = \big(\prod_{i \ne j_1,j_2} S_i\big) \oplus H$ for some (not necessarily distinct) integers $1 \le j_1 \leq j_2 \le r$ and some two-sided ideal $H$ of $A$ that is maximal with respect to being contained in $J(A)$. We first want to count the possibilities for $H$. 

Observe that $B^{\times}$ is a maximal subgroup of $A^{\times}$. 
Then $A^{\times}$ acts primitively by left-multiplication on the set of left cosets $A^{\times}/B^{\times}$ (see for instance $1.7(b)$ of \cite{C}). In other words, $A^{\times}/B^{\times}$ is a primitive (left) $A^{\times}$-space. Applying $1.3$ of \cite{C} then tells us that the conjugacy class of $B^{\times}$ in $A^{\times}$, and hence $H$, is uniquely determined by the isomorphism class of $A^{\times}/B^{\times}$ as an $A^{\times}$-space. 

Consider the (non-unital) quotient algebra $V:=J(A)/H$. We equip $V$ with the structure of an $A$-bimodule under the action $v \mapsto ava'$ for $a, a' \in A$ and $v \in V$. Note that $V$ is a simple $A$-bimodule since $B$ is a maximal subalgebra of $A$, and hence $J(A)$ acts trivially on $V$ (on both the left and the right) by Nakayama's lemma. So $V^2=0$. 

Next consider the quotient algebra $A/C=:\overline{A}$ and the natural projection $\rho:A \to \overline{A}$. Let $\overline{S}$ be a maximal semisimple subalgebra of $\overline{A}$. Since our field $k$ is perfect, $\overline{A}/J(\overline{A}) \cong \overline{S}$ is separable. Observe that $V$ is isomorphic to $J(\overline{A})$ as a (non-unital) algebra. Then applying Proposition $11.7$ of \cite{Pi} (a version of Wedderburn's Principal theorem) gives us a semidirect product of algebras $\overline{A} = V \rtimes \overline{S}$ (that is, $(v,s)(v',s')=(vs'+sv',ss')$ for all $v,v \in V$ and $s,s' \in \overline{S}$, and $(0,1)$ is the unity element). 

The core of $\overline{A}$ is trivial, and so $\overline{S}$ is isomorphic to either $\smash{S_{j_1}}$ (if $j_1=j_2$) or to $\smash{S_{j_1} \times S_{j_2}}$ (if $j_1 \neq j_2$). One can interpret $V$ as the natural left $\smash{S_{j_1} \otimes (S_{j_2})^{op}}$-module (refer to $\S 10.1$ of \cite{Pi}). It follows that there are at most $r^2$ possibilities for the isomorphism class of $\overline{A}$.

Taking the respective groups of units of $\overline{A} = V \rtimes \overline{S}$ gives us a semidirect product of groups $\smash{\overline{A}^{\times} = V \rtimes \overline{S}^{\times}}$ (considering $V$ to be its additive group). 
The natural projection $\rho:A \to \overline{A}$ induces a (left) action of $A^{\times}$ by permutations on $V$ as follows. Let $v \in V$ and $a \in A^{\times}$, say $\rho(a)=(v',s)$ for $v' \in V$ and $\smash{s \in \overline{S}^{\times}}$, and define $a \cdot v:=svs^{-1}+v's^{-1}$. It is then easy to see that the map $V \to A^{\times}/B^{\times}$ given by $x+H \mapsto (x+1)B^{\times}$ for $x \in J(A)$ is an isomorphism of $A^{\times}$-spaces. 


The isomorphism class of $V$ as an $A^{\times}$-space is uniquely determined by the isomorphism class of $\smash{\overline{A}^{\times}}$ along with a $1$-cocycle $A^{\times} \to V$, which arises from a derivation $\delta:A \to V$. 
Certainly $\smash{\overline{A}^{\times}}$ is determined up to isomorphism by $\overline{A}$ and $\delta:A \to V$ is determined by its values on the generators of $A$. By assumption, $|V|=n$. In summary, the number of possibilities for $H$ is bounded above by $r^2n^{d(A)}$.

For a given $H$, by Malcev's contribution to the Principal Theorem and since $B \cap J(A) = J(B) = H$, there are precisely $|J(A)/H|= n$ possibilities for $B$. So there are at most $r^2n^{d(A)+1}$ possibilities for $B$ of type $(T3)$. This completes the proof.
\end{proof}

Let $x = (x_d)_{d\in \N}$ be a sequence of elements of $A$ that are chosen randomly, uniformly and independently. Define a random variable $\tau_A$ by
\[
\tau_A = \min\{d \ge 1 \hspace{0.5mm}|\hspace{0.5mm} \langle x_1, . . . , x_d \rangle = A\} \in \N \cup \{+\infty\}.
\]
Recall that $E(A)$ is the expected number of random elements of $A$ chosen uniformly and independently which generate $A$. Observe that
\begin{equation}\label{nice}
E(A) = \sum_{d \geq 1} dP(\tau_A =d) = \sum_{d \ge 1} \big(\sum_{c \ge d} P(\tau_A =c)\big) = \sum_{d \ge 0}\big(1-P(A,d)\big).
\end{equation}

Recall the definitions
\[
M(A) = \sup_{n > 1} \log{m_n(A)}/ \log{n}
\]
and, for any real number $\eta \ge 1$,
\[
V_{\eta}(A) = \min \{ d \ge 1: P(A,d) \ge \eta^{-1} \}.
\]

Let $\zeta$ denote the Riemann zeta function.

\medskip

\noindent{\it Proof of Theorem \ref{PSAA}.}

\begin{proof} We first prove $(i)$. Observe that $m_n(A) =0$ for $n < q$ (or indeed if $n < m(A)$). We claim that \begin{equation}\label{kik} m_n(A) \leq 2r^2n^{d(A)+1}.\end{equation} Assuming that $(\ref{kik})$ holds, we obtain \begin{align*}
M(A) & = \max_{n >1} \log{m_n(A)}/\log n \\
& \le \max_{n \geq q} \log (2r^2n^{d(A)+1})/ \log n \\
& \le 2\log_q r + d(A)+2.
\end{align*} It remains to show that the inequality $(\ref{kik})$ holds.

Recall from Proposition \ref{mn} that
\[
m_n(A) \le n\big(6.93r+r(r-1)/2+r^2n^{d(A)} \big).
\]
It follows immediately that $(\ref{kik})$ holds, unless (possibly) if $d(A) \leq 1$, or if $d(A)=2$, $n=q=2$ and $r=1$.

If $d(A)=0$ then $A \cong k$ and $m_n(A)=0$, and of course $(\ref{kik})$ holds.

Next assume that $d(A)=1$. Let $i \in \{1,...,r\}$. Since $d(A)=1$, Lemma \ref{smalld}$(ii)$ tells us that $S_i$ is a field. It then follows from Theorem \ref{classthm1} that any maximal subalgebra of $S_i$ is of type $(S3)$, and hence for any given $n$ there is at most one index $n$ maximal subalgebra of $S_i$. That is, $m_n(S_i) \leq 1$. We can use this to refine the proof of Proposition \ref{mn}, and obtain the sharpened inequality $$m_n(A) \le r+n\big(r(r-1)/2+r^2n \big).$$ It follows that $(\ref{kik})$ holds.

Finally, assume that $d(A)=2$, $n=q=2$ and $r=1$. Since $n=q$, the invariant $m_n(A)$ is counting the codimension $1$ subalgebras of $A$. Note that $S$ is simple since $r=1$, and so $A$ has no maximal subalgebras of type $(T2)$.

Assume for the moment that $S \cong k$. Then $A$ has no maximal subalgebras of type $(T1)$. By the proof of Proposition \ref{mn}, $A$ has at most $8$ maximal subalgebras of type $(T3)$. Hence $m_n(A) \le 8$ by Theorem \ref{classthm2}, so certainly $(\ref{kik})$ holds.

Now assume that $S \not\cong k$. Then any maximal subalgebra of type $(T3)$ of $A$ has codimension strictly greater than $1$. Codimension $1$ subalgebras of type $(T1)$ of $A$ are in bijection with codimension $1$ subalgebras of $S$, of which there are none unless $S \cong M_2(q)$ (in which case there are $3$ of them) 
or $S \cong \F_{q^2}$ (in which case there is $1$ of them). So $m_n(A) \le 3$ by Theorem \ref{classthm2}, and again $(\ref{kik})$ holds.

This completes the proof of part $(i)$.

$(ii)$. Let $d = \lceil M(A) + 2.02 \rceil$. It is immediate from the definition that $m_n(A) \le n^{M(A)}$ for all $n \ge 2$. Then we have
\[
1-P(A,d) \le \sum_{n \ge 2} m_n(A)n^{-d} \le  \sum_{n \ge 2} n^{-2.02} = \zeta(2.02) -1 < 1 - e^{-1}.
\]
Hence $P(A,d) > e^{-1}$, so $V(A) \le d$, as required. This gives us the upper bound.

Now let $d = V(A)$. Then $P(A,d) \ge e^{-1}$, and it follows from equation $(\ref{strong})$ that
\[
m_n(A) \le 6.93 e \cdot n^{V(A) + 1}.
\]
This implies that
\[
M(A) = \sup_{n > 1} \log{m_n(A)}/ \log{n} \le \log(6.93e) +V(A) + 1 < V(A) + 5.24.
\]
Since $V(A)$ is an integer, we have $V(A) \geq \lceil M(A)- 5.24\rceil$.

$(iii)$. Let $\eta \geq 1$ be a real number and set $d = V_{\eta}(A)$. By the same argument as in the proof of $(ii)$, we see that
\[
m_n(A) \le 6.93 \eta \cdot n^{V_{\eta}(A) + 1}.
\]
This implies
\begin{equation}\label{Maa}
M(A) = \sup_{n > 1} \log m_n(A) / \log n < V_{\eta}(A)+\log \eta + 3.80.
\end{equation}

Now consider the case where $\eta = 2^i$ for some positive integer $i$. Then $$M(A) < V_{2^i}(A)+i+3.80$$ by equation $(\ref{Maa})$. Denote $\alpha:=\lceil M(A)-4.80\rceil$. In particular, if $d=\alpha-i$ then $P(A,d) < 2^{-i}$. Combining this with equation $(\ref{nice})$ gives us
\[
E(A) \ge \sum_{d = 0}^{\alpha-1}\big(1-P(A,d)\big) > \sum_{d = 0}^{\alpha-1} \big(1-2^{d-\alpha}\big) \geq \alpha-1.
\]

This establishes the lower bound of $E(A)$. It remains to prove the upper bound.

Denote $l:=\lceil M(A) \rceil$. Since $m_n(A) \le n^l$, we have
\[
1-P(A,d) \le \sum_{n \ge 2} m_n(A)n^{-d} \le \sum_{n \ge 2} n^{l-d}.
\]

Denote $\beta := d-l$. Combining this with $(\ref{nice})$ gives us
\begin{align*} E(A) & \le l+2 + \sum_{d \ge l+2} \big(1-P(A,d)\big)  \\
& \le l+2 + \sum_{\beta \ge 2} \big( \sum_{n \ge 2} n^{-{\beta}}\big) \\
& \le l+2+ \sum_{n \ge 2} (\zeta(n)-1) \\
& = l+3. 
\end{align*}

This completes the proof of $(iii)$.
\end{proof}

\end{document}